\documentclass[a4paper,12pt]{article}

\usepackage{amsmath,amsthm,amssymb,sansmath,color,geometry,multicol,graphicx, float,color,authblk, hyperref,bbm,stmaryrd} 
\usepackage[english]{babel}
\usepackage[utf8]{inputenc}
\newtheorem{thm}{Theorem}
\newtheorem{prop}[thm]{Proposition}

\newtheorem{lem}[thm]{Lemma}
\newtheorem{cor}[thm]{Corollary}
\newtheorem{hyp}{Assumption}

\newcommand{\cco}{\llbracket}
\newcommand{\ccf}{\rrbracket}
\newcommand{\po}{\left(}
\newcommand{\pf}{\right)}
\newcommand{\co}{\left[}
\newcommand{\cf}{\right]}
\newcommand{\ao}{\left\{}
\newcommand{\af}{\right\}}
\newcommand{\R}{\mathbb R}
\newcommand{\Sd}{\mathbb S^{d-1}}
\newcommand{\E}{\mathbb E}
\newcommand{\n}{\nabla}

\title{Piecewise deterministic simulated annealing}
\author{Pierre Monmarché}

\begin{document}

\maketitle
\abstract{Given an energy potential on the Euclidian space, a piecewise deterministic Markov process is designed to sample the corresponding Gibbs measure. In dimension one an Eyring–Kramers formula is obtained for the exit time of the domain of a local minimum at low temperature, and a necessary and sufficient condition is given on the cooling schedule in a simulated annealing algorithm to ensure the process converges to the set of global minima. This condition is similar to the classical one for diffusions and involves the critical depth of the potential. In higher dimensions a non optimal sufficient condition is obtained.}

\section{Introduction and main results}

\subsection{Simulated annealing}

The simulated annealing algorithm is a classical stochastic optimization algorithm, which can be seen as a descent algorithm perturbed by random locally counter-productive moves to escape from non-global minima. More precisely, consider an  energy potential $U$  on the Euclidian space such that $e^{-U}\in L^1(\R^d)$. The mass of the Gibbs law associated to $\frac1\varepsilon U$,
\[\mu_\varepsilon(dx) = \frac{e^{-\frac{U(x)}\varepsilon}}{\int_{\R^d} e^{-\frac{U(y)}\varepsilon}dy}dx,\]
  concentrates on the set of global minima of $U$ as $\varepsilon>0$, called the temperature, goes to zero. At fixed temperature, $\mu_\varepsilon$ can be sampled by a Markov Monte-Carlo procedure, namely it can be approximated by the law of $X_T^{\varepsilon}$ where $\po X_t^{\varepsilon}\pf_{t\geq 0}$ 
 is an ergodic Markov process whose invariant law is $\mu_\varepsilon$ and $T$ is large enough so that the process is close to equilibrium at time $T$. A usual choice for $X^\varepsilon$ would be the Fokker-Planck diffusion, that is the solution of
 \begin{eqnarray}\label{EqFokkerPlanck}
  dX^\varepsilon_t & = & - \nabla U\po X^\varepsilon_t\pf dt + \sqrt{2\varepsilon} dB_t,
 \end{eqnarray}
where $B$ is a Brownian motion. In the simulated annealing
algorithm the temperature $\varepsilon_t$ decays over time, so that the solution of this stochastic differential equation becomes an inhomogeneous Markov process $(X_t^{\varepsilon_t})_{t\geq 0}$. If the system cools slowly, \emph{i.e.} if $t\mapsto \varepsilon_t$ goes to zero slowly enough, then the process has enough time to explore the space and approach equilibrium, so that the law of $X_t^{\varepsilon_t}$ gets and stays close to its ``instantaneous'' invariant law, $\mu_{\varepsilon_t}$. In particular, the mass of the law of $X_t^{\varepsilon_t}$ goes to the set of the global minima of $U$ as $t$ goes to infinity. However, if the system is abruptly frozen, in other word if the decay $t\mapsto \varepsilon_t$ is too fast, the process will have a non-zero probability to be
trapped in local minima.

There is a broad literature on this question, both theoretical and practical, and we refer to \cite{Kalivas} for an introduction. A key phenomenon in the analysis of this algorithm (and of many stochastic algorithms, indeed) is metastability (cf. \cite{LelievreMetastable}): at low temperature, the process spends a lot of time trapped in a neighbourhood of local minima, so that mixing - \emph{i.e.} convergence to equilibrium (at fixed temperature) - is very slow. This yields extremely slow theoretical cooling schedules $t\mapsto \varepsilon_t$ to ensure the process converges in probability to the set of global minima of $U$. For instance in the Fokker-Planck diffusion case, it is well-known (see \cite{Holley} among others) that in order for the process to converge in probability to any neighbourhood of the global maxima of $U$, $\varepsilon_t$ should be of order at least $\frac{E^*}{\ln(1+t)}$ where $E^*$ is the critical depth of the potential, a constant that depends on $U$ which will be defined below.

One line of inquiry to improve the algorithm is then to look for other dynamics than \eqref{EqFokkerPlanck}, which would have more inertia and thus would escape more easily from local traps. One of the main example is the kinetic Langevin dynamics, studied in \cite{Lelievre2006,MonmarcheRecuitHypo}. Among other possibilities, the reversible dynamics \eqref{EqFokkerPlanck} is perturbed in \cite{Lelievre2012} with a divergence-free drift; or processes with more general memories than kinetic ones are considered in \cite{GadatPanloup2013}.

In this work we propose a Markov process $(X_t,Y_t)_{t\geq 0}$ on $\R^d\times\Sd$ with the following properties :
\begin{itemize}
 \item $dX_t = Y_t dt$.
 \item the process is ergodic (in the sense that for all initial condition, its law converges to a unique invariant law $\mu$), and the first marginal of $\mu$ is prescribed as  $e^{-U(x)}dx$.
 \item $Y_t$ is a jump process.
\end{itemize}
The first property means $(X_t,Y_t)$ is a kinetic process, such as the kinetic Langevin one, and we call $X$ the position and $Y$ the velocity. The second one means it serves its intended purpose. Finally the fact that $Y_t$ is a jump process on the sphere makes the whole process very easy to sample on a computer. More precisely between two times of jump of $Y_t$, $(X_t,Y_t)$ is completely deterministic. Such piecewise deterministic Markov processes (PDMP) have recently attracted much attention in various fields, since they are a simple  alternative to diffusions to model stochastic systems (see \cite{Krell} and references within for an overview).

In dimension greater than 1 there are many ways to meet the above requirements and thus in a first instance we will focus on the case $d=1$, for which the possibilities are more limited. 


\subsection{The one-dimensional process}\label{SectionDefRTP}

Our aim is to define a Markov process $(X_t,Y_t)$ with the properties above. In dimension one, the velocity is either $1$ or $-1$. If the process goes twice through the same state $(x,y)$, necessarily it has had to make a U-turn in the meanwhile and come back the other way, hence to visit $(x,-y)$. On average (in time), $(x,1)$ and $(x,-1)$ are thus equally visited. Ergodicity then implies the invariant law is necessarily a product measure whose second marginal is uniform on  $\{\pm1\}$. 

Recall that the semi-group $(P_{s,t})_{t\geq s\geq0}$ and the infinitesimal generator $\po L_t\pf_{t\geq 0}$ associated to the process are defined as
\[P_{s,t} f(x,y) = \E\co f(X_t,Y_t)\ | \ (X_s,Y_s) = (x,y)\cf\]
for all bounded $f$ and
\[L_tf = \underset{h\rightarrow 0}\lim \frac{P_{t,t+h} f - f}{h}\]
whenever this limit (say in the uniform norm sense) exists. When the process is homogeneous (in our case, it means when the temperature is constant, which is the case for now), we only write $P_t = P_{0,t}$ and $L=L_t$. We will describe the dynamics through the generator. 

In dimension one, 
the only possibility when a jump occurs is to transform the velocity into its opposite. This yields an infinitesimal generator of the form 
\begin{eqnarray}\label{EqGenedebut}
Lf(x,y) = y f'(x,y) + \lambda(x,y) \po f(x,-y) - f(x,y)\pf
\end{eqnarray}
where the rate of jump $\lambda$ is a non-negative function. In this case, the law $\mu$ is invariant iff $\int L f d\mu = 0$ for all $f$. This is equivalent to
\[y U'(x) = \lambda(x,y) - \lambda(x,-y),\]
which is satisfied if and only if $\lambda$ is of the form $\lambda(x,y) = \frac12\po y U'(x) + a(x)\pf$ for some function $a(x)$ (note that necessarily $a(x) = \lambda(x,y) + \lambda(x,-y)$). The  non-negativity of $\lambda$ implies $a(x) \geq |U'(x)|$. When this is indeed an equality,  $\lambda(x,y) = \po yU'(x)\pf_+$ (where $(g)_+$ denotes the positive part of $g$, equal to $g$ if $g>0$ and $0$ else): this is the choice that minimizes the rate of jump, namely the dissipative behaviour of the process. On the other hand it is convenient from the simulation viewpoint. Indeed, it implies that while $Y_tU'(X_t) \leq 0$, or in other words while the process is going down the potential, no jump is allowed. On the contrary if $Y_tU'(X_t) > 0$ the next time $T$ of jump will be such that
\[E:=\int_0^T y U'(x+ys) ds = U(x+yT) - U(x)\]
has an exponential law with mean 1, which we  denote by $\mathcal E(1)$. Thus we only need to compute the potential along the trajectory, and to simulate a Poisson process. More precisely, let $(E_k)_{k\geq 0}$ be a family of i.i.d. variables with exponential law and set $(X_0,Y_0)=(x_0,y_0)$ and $T_0=0$. Suppose the process is already constructed up to a jump time $T_i$, $i\geq 0$, and is independent from $(E_k)_{k\geq i}$ up to $T_i$. Let
\begin{eqnarray}\label{DefiTavecExpo}
 T_{i+1} & = & \inf\left\{t > T_i,\ \int_{T_i}^t \lambda\po X_{T_i} + Y_{T_i}(s-T_i),Y_{T_i}\pf ds \geq E_i\right \}\\
 X_s & = &  X_{T_i} + Y_{T_i}(t-T_i)\hspace{20pt}\text{ for }s\in[T_i,T_{i+1}]\notag\\
 Y_s & =&  Y_{T_i} \hspace{87pt}\text{ for } s\in[T_i,T_{i+1})\notag\\
 Y_{T_{i+1}} & = & - Y_{T_i}.\notag
\end{eqnarray}
This define a Markov process $(X_t,Y_t)$ with generator \eqref{EqGenedebut}.

In the literature such a process, which belongs to the larger class of switched PDMP (\cite{BLBMZ2}) goes sometimes by the name of (integrated) telegraph process (\cite{Ratanov,Monmarche2013}). It can be seen as the continuous limit of persistent walks (\cite{DiaconisMiclo,Volte-Face}), which were already studied as a possible alternative to reversible walks to sample discrete Gibbs measures, and it is reminiscent of the so-called Hit-and-Run sampler (\cite{HitandRun}). It has already been studied to model the motion of the bacterium \emph{Escheria coli} (\cite{Calvez,Fontbona2010,Othmer}) and is called in this context a velocity jump process, which is the name we are going to use since it still makes sense in a metastable context and in dimension greater than 1 (see Section \ref{SectionDefDim}).

\subsection{Main results in dimension one}

As a first step we will consider the velocity jump process $(X_t^\varepsilon,Y_t^\varepsilon)_{t\geq0}$ on $\R\times\{\pm1\}$ at low (but fixed) temperature $\varepsilon>0$, namely with generator
\begin{eqnarray}
 L f(x,y) & = & y f'(x,y) + \frac1\varepsilon\po y U'(x)\pf_+\po f(x,-y) - f(x,y)\pf.\label{EqGenerateur}
\end{eqnarray}
We want to understand how long it takes for the process to escape from a local minimum. Let $U$ be a double-well potential, namely a Morse function with two local minima, denoted by $x_0\leq x_2$, a local maximum $x_1\in(x_0,x_2)$, and which is convex outside $(x_0,x_2)$ and goes to infinity as $|x|\rightarrow\infty$. Recall that $U$ is said to be a Morse function if all its critical points are non-degenerate; in other words if $\po U'(x) =0\pf\Rightarrow\po U''(x)\neq 0\pf$. Suppose $(X_0^\varepsilon,Y_0^\varepsilon) = (x_0,-1)$, and let 
\begin{eqnarray*}
 \tau & = & \inf \left\{t>0, X_t^\varepsilon =x_1\right\}
\end{eqnarray*}
be the first hitting time of $x_1$ (note that, contrary to a diffusion which may fall back, when $X^\varepsilon$ reaches $x_1$ it deterministically leaves $[x_0,x_1]$ and falls down to $x_2$). Then the energy barrier to overcome in order to leave $[x_0,x_1]$ is $U(x_1) - U(x_0)$. We will prove what is usually called an Eyring–Kramers formula (or an Arrhenius law):
\begin{thm}\label{ThmCVexpoDimUn}
For the velocity jump process with generator \eqref{EqGenerateur} starting at $(X_0^\varepsilon,Y_0^\varepsilon) = (x_0,-1)$ in the double-well potential $U$,
  \begin{eqnarray*}
  \E \co \tau\cf & = &  \sqrt{\frac{8\pi \varepsilon}{ U''(x_0)}}  e^{\frac{U(x_1) - U(x_0)}{\varepsilon}} \po 1 + \underset{\varepsilon \rightarrow 0}o (1)\pf,\\
  \mathbb P\po \tau\geq t \E \co \tau\cf \pf & \underset{\varepsilon \rightarrow 0}\longrightarrow & e^{-t}.
 \end{eqnarray*}
\end{thm}
This can be compared to the case of a Fokker-Planck diffusion $Z_t^\varepsilon$ with generator 
\begin{eqnarray*}
 L_{di} f(x) &  =&  -U'(x) f'(x) + \varepsilon f''(x),
 \end{eqnarray*}
 which has been studied in much more general settings. Let $\eta>0$ be small and
 \begin{eqnarray*}
\tau_{di} & = & \inf \left\{t>0, Z_t^\varepsilon = x_1+\eta\right\}.
\end{eqnarray*}
The work of Bovier \emph{$\&$ al} applies here and yields:
\begin{thm}[from \cite{Bovier1,Bovier2}]
For the Fokker-Planck diffusion starting at $Z^\varepsilon_0 = x_0$ in the double-well potential $U$,
 \begin{eqnarray*}
  \E \co \tau_{di}\cf & = & \frac{2\pi e^{\frac{U(x_1)-U(x_0)}{\varepsilon}}}{\sqrt{|U''(x_1)|U''(x_0)}} \po 1 + \underset{\varepsilon \rightarrow 0}o (1)\pf,\\
  \mathbb P\po \tau_{di} \geq t \E \co \tau_{di}\cf \pf & \underset{\varepsilon \rightarrow 0}\longrightarrow & e^{-t} .
 \end{eqnarray*}
\end{thm}
\textbf{Remarks:}
\begin{itemize}
\item Both processes samples the same Gibbs law. 
To sample the diffusion, one needs to generate a Brownian motion, while the PDMP can be constructed from a sequence of independent exponential variables $(E_k)_{k\geq 0}$.
 \item The velocity jump process moves at constant speed. In particular it takes a constant time (constant in the sense it does not depends on $\varepsilon$) to pass through the interval $(x_1-\eta,x_1+\eta)$, whose probability under the law $\mu_\varepsilon$ is of order
 \[\exp \po {-\frac{1}{\varepsilon}}\po \min(U(x_1-\eta),U(x_1+\eta))-\min\po U(x_0),U(x_2)\pf\pf\pf.\]
 Since the ratio between the average times spent in this interval and outside of it should be of this order, it means the time between two crossing of $(x_1-\eta,x_2+\eta)$ needs to be of order the inverse of this probability, which explains the exponential factor of $\E \co \tau\cf$ had to be expected.
 \item The fact $U''(x_1)$ does not appear in the PDMP case is also natural. Indeed, the probability that the process starting at $(x_0,1)$ reaches $x_1$ in one shot (\emph{i.e.} before coming back to $(x_0,-1)$) depends only on $U(x_1)-U(x_0)$, and not on the local geometry of the potential near $x_1$. On the contrary the process stays mainly in the neighbourhood of $x_0$, so that $U''(x_0)$ does intervene. If $U$ were flat, for instance $U= d\po 1 - \mathbb 1_{[x_0-1,x_0+1]}\pf$ for some $d>0$, an adaptation of the proof of Theorem \ref{ThmCVexpoDimUn} would yield $\E\co \tau\cf \simeq 2 e^{\frac d \varepsilon}$.
 \item Both $\E \co \tau\cf$ and $\E \co \tau_{di}\cf$  have the same exponential order.
\end{itemize}

With Theorem \ref{ThmCVexpoDimUn} in mind we then turn to the study of the inhomogeneous process $(X_t^{\varepsilon_t},Y_t^{\varepsilon_t})$ with generator
\begin{eqnarray}
 L_t f(x,y) & = & y f'(x,y) + \frac1{\varepsilon_t}\po y U'(x)\pf_+\po f(x,-y) - f(x,y)\pf.\label{EqGenerateurDependdeT}
\end{eqnarray}
\begin{hyp}\label{HypoUepsi}
\begin{itemize}
\item The cooling schedule $t\mapsto \varepsilon_t>0$ is non-increasing and goes to 0.
\item The potential $U$ on $\R$ is a  smooth Morse function  with a finite number of local extrema (one of which at least is a non-global minimum), unbounded and convex at infinity.
\end{itemize}
\end{hyp}
 We say that $z$ is reachable from $x$ at height $V$ if $\max\{U\po x+t(z-x)\pf,\ t\in[0,1]\}\leq V$ and we call the depth of a local minimum $x$ the smallest $V$ such that there exist a $z$ with $U(z)<U(x)$ which is reachable from $x$ at height $U(x)+V$ (the depth of a global minimum is set to $+\infty$). The critical depth of $U$, denoted by $E^*$, is then defined as the maximal among the depths of all local minima of $U$ which are not global minima (see Fig. \ref{FigureProfondeur}).

Adapting to our settings the work of Hajek (\cite{Hajek}) on simulated annealing on a discrete space, we will prove the following :

\begin{thm}\label{TheoremNS}
Let $U$ and $(\varepsilon_t)_{t\geq 0}$ satisfy Assumption \ref{HypoUepsi} and consider the process $(X_t^{\varepsilon_t},Y_t^{\varepsilon_t})$ with generator \ref{EqGenerateurDependdeT} and any initial condition. We have the following:
\begin{enumerate}
 \item If $S$ is a neighbourhood of all local minima of $U$, then
 \[\underset{t\rightarrow\infty}\lim \mathbb P\po X_t^{\varepsilon_t} \in S\pf =1.\]
 \item If $S$ is a neighbourhood of all local minima of depth $E$, and its complementary $S^c$ is a neighbourhood of all other local minima,
 \[\underset{t\rightarrow\infty}\lim \mathbb P\po X_t^{\varepsilon_t} \in S\pf =0 \hspace{25pt}\Leftrightarrow \hspace{25pt} \int_0^\infty  \po \varepsilon_{s}\pf^{-\frac12}e^{-\frac{E}{\varepsilon_{s}}} ds = \infty\]
 \item   As a consequence,
  \[\forall \delta >0\ \underset{t\rightarrow\infty}\lim \mathbb P\po U\po X_t^{\varepsilon_t}\pf < \underset{\R}\min U + \delta\pf =1 \hspace{25pt}\Leftrightarrow \hspace{25pt}  \int_0^\infty  \po \varepsilon_{s}\pf^{-\frac12}e^{-\frac{E^*}{\varepsilon_{s}}} ds = \infty\]
\end{enumerate}
\end{thm}

\begin{figure}
 \centering
 \includegraphics[scale=0.5]{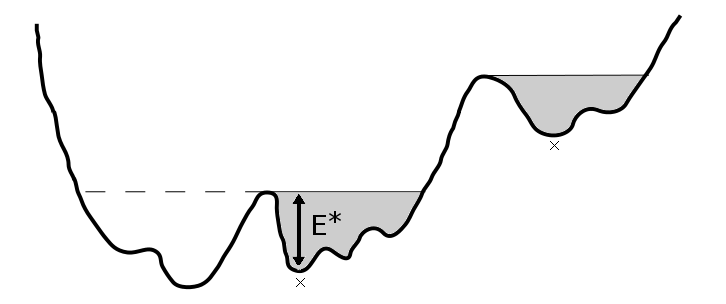}
 \caption{The depth of a local minima, and the critical depth.}\label{FigureProfondeur}
\end{figure}

In particular if the cooling schedule is of the parametric form
\[\varepsilon_t = \frac{c}{\ln(1+t)}\]
with $c>0$, the algorithm succeed (\emph{i.e.} the process converges in probability to any neighbourhood of the global minima of $U$) if and only if $c\geq E^*$.

It is somehow a negative result: it means this velocity jump process does not allow faster cooling schedules than the classical reversible diffusion one. On the other hand it is somehow a positive result, since it allows cooling schedule as fast as the diffusion \emph{and} it is easier and faster to compute numerically. Finally, positive or negative, it is above all a theoretical result. In practice the simulation is done in a finite time horizon; a context where the theoretical logarithmic schedules are far from efficient. The next step of the study should thus be to give non-asymptotic results in the spirit of the work of Catoni (\cite{Catoni}).

\bigskip

Before proceeding to the proofs, here is a remark about our method. The extreme simplicity of the motion permits an elementary analysis. We could have tried to use functional inequalities tools, such as in \cite{Holley,Miclo92} (see also \cite{logSob} for a general introduction). Indeed, since $U$ is assumed to be convex at infinity, it is known the associated Gibbs measure satisfies a spectral gap (or Poincar\'e) inequality, with a well-understood asymptotic of the constant. But since the \emph{carr\'e du champs} of the velocity jump process is not the square of the gradient, it is not clear whether it gives any information on the way the process relaxes to equilibrium. This is a typical problem in the field of hypocoercivity, and indeed our PDMP has been studied in \cite{Calvez,Monmarche2013} from this viewpoint. But in both works the rate of jump is assumed to be bounded from below by a positive constant, which means there is at all time a residual randomness. This is not the case with our minimal choice $\lambda(x,y) = \po yU'
(x) \pf_+$ and in this sense our process is quite 
degenerate among degenerate processes. 

However, if on the one hand some ideas are missing to treat this degenerate situation with hypocoercive tools, on the other hand, hopefully, once precisely understood thanks to elementary analysis, this process may be a good benchmark to investigate several hypocoercive questions, such as the relationships between functional inequalities, gradient estimates and Wasserstein convergence.

Finally note that, even if the pathwise strategy we will adapt from Hajek is very close to the Freidlin and Wentzel approach \cite{FreidlinWentzell}, the latter theory would yield slightly less precise results since for continuous-time processes it only deals with the large deviation scaling (namely with the asymptotic of quantities such as $\frac1t\ln \mathbb P(X_t^\varepsilon \in A)$ rather than  $\mathbb P(X_t^\varepsilon \in A)$).

\subsection{Definition and results in any dimension}\label{SectionDefDim}

The interest of the simulated annealing algorithm appears in large dimension, and so we now define a suitable piecewise deterministic Gibbs sampler in this context. We call  velocity jump process the Markov process on $\R^d\times\mathbb S^{d-1}$ whose generator is
\begin{eqnarray*}
L f(x,y) &=   & y.\n_x f(x,y) + \po y.\n_x U(x)\pf_+ \po f(x,y^*) - f(x,y) \pf + r \po \int f(x,z) dz - f(x,y)\pf,
\end{eqnarray*}
where $r>0$ is a parameter, $dz$ denotes the uniform measure on $\Sd$ and
\[ y^* = y - 2 \po y. \frac{\nabla_x U(x)}{|\nabla_x U(x)|}\pf\frac{\nabla_x U(x)}{|\nabla_x U(x)|}\]
(the explanations of this definition is postponed to the end of this section). As we will see, the measure $e^{-U(x)}dx \times dy$ is invariant for this process. A trajectory is defined in a similar manner that in dimension 1, except that there are now two different clocks:
  \begin{eqnarray*}
 T  & = & \inf\left\{t > 0,\ \int_{0}^t \po Y_{0} .\nabla U\po X_{0} + Y_{0}s\pf\pf_+ ds \geq E_1\right \},\\
 S  & = & \frac1r E_2,
\end{eqnarray*}
where $E_1,E_2$ are independent standard exponential random variables. The process evolves deterministically according to $dX = Y$ and $dY =0$ up to time $T\wedge S$, at which its velocity jumps from $Y$ to $Y'$ uniformly sampled on the sphere if $T\wedge S=S$, or to   $Y^*$ if $T\wedge S=T$. This latter case means that only the part of $Y$ which is parallel to $\nabla_x U(X)$ jumps (to its opposite), while the orthogonal part is left untouched. An interpretation is that when it jumps, the process is deterministically reflected according to optical laws - or as a billiard - on the level set of $U$ it has reached (see Fig. \ref{Figrebond}).

\begin{figure}
 \centering
\includegraphics[scale=0.25]{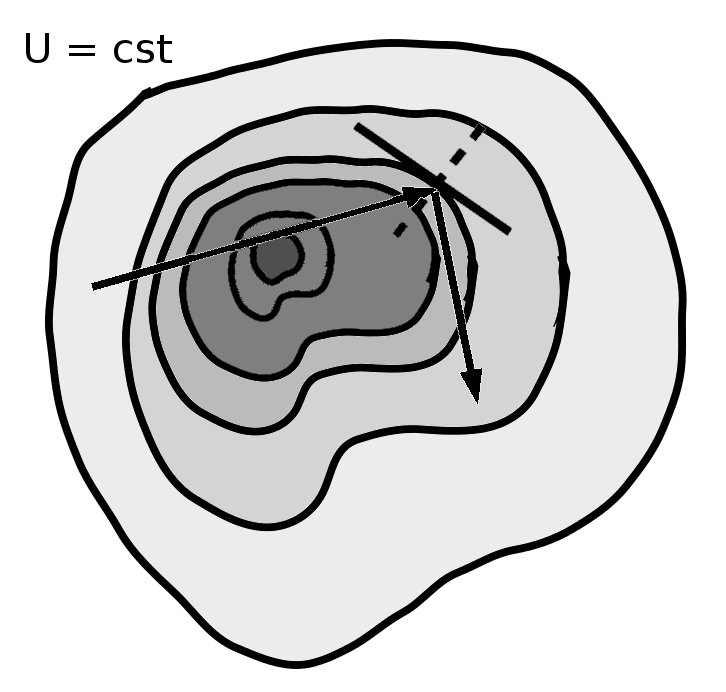}
\caption{At a jump time the process bounces off the level sets of $U$.}\label{Figrebond}
\end{figure}

\bigskip

We restrict the study to a compact case: suppose $U$ is smooth and 1-periodic in the sense $x-z\in \mathbb Z^d \Rightarrow U(x)=U(z)$, let $\mathbb T^d = (\R /\mathbb Z)^d$, consider a cooling-schedule $(\varepsilon_t)_{t\geq 0}$ and the inhomogeneous Markov process $Z^\varepsilon=\po X_t^{\varepsilon_t},Y_t^{\varepsilon_t}\pf_{t\geq 0}$ on $\mathbb T^d \times \mathbb S^{d-1}$ with generator
\begin{eqnarray}\label{EqGenerateurDimD}
\lefteqn{L_t f(x,y) =}\notag\\ 
&   & y.\n_x f(x,y) + \frac{\po y.\n_x U(x)\pf_+}{\varepsilon_t} \po f(x,y^*) - f(x,y) \pf + r \po \int f(x,z) dz - f(x,y)\pf.
\end{eqnarray}
Recall the associated semi-group acts on probability laws $\mathcal P\po \mathbb T^d \times \mathbb S^{d-1}\pf$ by
\[\mu P_{s,t}  = \mathcal L\po (Z_t)\ | \ \mathcal L\po (Z_s)= \mu \pf\pf.\]
 Let  $\nu_t = e^{-\frac1{\varepsilon_t} U(x)}d x\times d y$ and
\[\|\mu_1 - \mu_2\|_{TV} = \underset{Z_1\sim \mu_1,Z_2\sim \mu_2}\inf \mathbb P\po Z_1\neq Z_2\pf\]
be the total variation distance.

\bigskip

First we establish an exponential relaxation to equilibrium when the temperature is fixed:
\begin{thm}\label{ThmDimDConstant}
There exist $c,\theta>0$ that depend only on $d,r$ and $U$ such that if $\varepsilon_t = \varepsilon>0$ is constant (so that $P_{0,t}=P_t$ is homogeneous and $\nu_t = \nu_0$ for all $t$), for all $t\geq0$ and all initial law $\mu$
\begin{eqnarray*}
 \| \mu P_{t} - \nu_0\|_{TV}   &\leq &   e^{-ce^{-\frac{\theta}{\varepsilon}} \po t- \sqrt d\pf} \left\|\mu-\nu_{ 0}\right\|_{TV}.
 \end{eqnarray*}
\end{thm}
Second we obtain a sufficient condition on a cooling schedule for the inhomogeneous process with generator given by \eqref{EqGenerateurDimD} to converge on probability to any neighbourhood of the set of global maxima of $U$:
\begin{thm}\label{ThmDimDRecuit}
There exists $\theta>0$ that depends only on $d$ and $U$ such that the following hold: let $\po\varepsilon_t\pf_{t\geq0}$ be any positive decreasing cooling schedule that goes to 0 as $t$ goes to infinity, and such that moreover for $t$ large enough
\[\partial_t \po \frac{1}{\varepsilon_t}\pf \leq \frac{1}{(\theta+\eta) t}\]
for some $\eta >0$. Then, for any $h>0$ and for any initial point $(X_0,Y_0)\in\mathbb T^d\times\mathbb S^{d-1}$,
\begin{eqnarray*}
\mathbb P\po U\po X_t^{\varepsilon_t}\pf > \min U + h\pf & \underset{t\rightarrow \infty}{\longrightarrow}& 0.
\end{eqnarray*}
In particular if $\varepsilon_t = \frac{\theta + \eta}{\ln t}$ for $t$ large enough, then there exists $C>0$ such that for all $t,h>0$ and $(X_0,Y_0)\in\mathbb T^d\times\mathbb S^{d-1}$,
\begin{eqnarray*}
\mathbb P\po U\po X_t^{\varepsilon_t}\pf > \min U + h\pf & \leq & C \po\frac{1}{t}\pf^{\frac{\min(h/2,\eta)}{\theta+\eta}}.
\end{eqnarray*}
\end{thm}
There is no reason not to expect that the smallest $\theta$ such that Theorems \ref{ThmDimDConstant} and \ref{ThmDimDRecuit} hold is $E^*$ the critical depth of the potential $U$ (see the remark after the proof of Theorem \ref{ThmDimDConstant} concerning the explicit bound we get). 
We believe that a sharper analysis, similar to our study in dimension 1 or to classical studies of the simulated annealing based on the Fokker-Planck diffusion, together with some ideas from the proof of Theorem \ref{ThmDimDRecuit} which are specific to the present process, could in fact enable us to reach $\theta = E^*$ (see the conclusive remark at the end of the paper). However we also think that somehow the additional technicalities is not worth the improvement of the result, at least in a first instance: indeed,  in an applied problem, $E_*$ is anyway unknown and out of reach.

The restriction to a compact space in Theorem \ref{ThmDimDRecuit}, which simplifies the study, should not be necessary. If for instance we suppose $U$ goes to infinity as $| x|\rightarrow\infty $  and $\varepsilon_t\rightarrow 0$ then at some time the process will stay forever trapped in a compact set, and the same arguments will work (the constant $C$ however will depend on the initial position). In Theorem \ref{ThmDimDConstant} if $U$ goes to infinity as $| x|\rightarrow\infty $ we can prove via Lyapunov techniques that the process always come back to a given compact set, but since it moves at constant speed, $\| \mu P_{t} - \nu_0\|_{TV} $ will not be controlled  by $\| \mu - \nu_0\|_{TV} $  uniformly on $\mu$; some moments of $\mu$ will be involved to account for the first hitting time of a given compact set.

\bigskip

\textbf{Remark on the definition of the multidimensional process.} We want to define a process with the specifications of Section \ref{SectionDefRTP}: $(X,Y)$ has to be a kinetic Markov process on $\R^d\times\mathbb S^{d-1}$ with a piecewise constant velocity. It remains to choose the jump rate and kernel of $Y$ with the constraint that the first marginal of the invariant measure has to be the Gibbs law associated to some potential $U$.

Recall that in dimension 1 an asset of the velocity jump process with minimal jump rate $\po yU'(x)\pf_+$ is that, in order to determine the next jump time, it is only necessary to compute $U(X_t)$ along the trajectory. In any dimension, this is still true with the rate $\lambda(x,y) = \po y.\n U(x)\pf_+$. Indeed in that case, while there is no jump,
\begin{eqnarray*}
\int_0^t \lambda(X_s,Y_s) ds & = & \int_0^t \po \partial_s \po U(X_s)\pf\pf_+ ds\\
&=& \left\{\begin{array}{lr}
0 &\text{while $U(X_t)$ decreases with $t$}\\
   U(X_t) - U(X_0)&\text{while $U(X_t)$ increases with $t$,}
\end{array} 
\right.
\end{eqnarray*}
and more generally it is the cumulated increases of $U$ along the trajectory since the last jump.  Note that on the contrary, in the case of the Fokker-Planck diffusion, $\nabla U(X_t)$ should be estimated all the time, possibly at a huge numerical cost in large dimension.

This choice of jump rate yields a generator of the form
\[L f(x,y) = y.\n_x f(x,y) + \po y.\n_x U(x)\pf_+ \po \int_{\Sd} f(x,z) p_{x,y}(z)\sigma(dz)  - f(x,y) \pf\]
with a transition kernel $p_{x,y}(z)$ on $\Sd$ which is still to be determined. Let $\nu = e^{-U(x)}dx \times dy$ where $dy$ is the uniform law on the sphere. In dimension greater than 1, this is an arbitrary - albeit the simplest - choice to ensure the first marginal of $\nu$ is the Gibbs law associated to $U$. Then $\nu$ is invariant for $L$ iff   $p_{x,y}(z)$ is a (weak) solution of
\begin{eqnarray*}
y.\n_x U(x) & =& \po y.\n_x U(x)\pf_+ - \int_{\Sd} \po z.\n_x U(x)\pf_+ p_{x,z}(y)\sigma(dz) 
\end{eqnarray*}
for a.e. $(x,y)\in\R^d\times\Sd$, or in other words
\begin{eqnarray*}
\int_{\Sd} \po z.\n_x U(x)\pf_+ p_{x,z}(y)\sigma(dz) &= &\po -y.\n_x U(x)\pf_+.
\end{eqnarray*}
 This is true if when $Y$ jumps, the part $Y.\nabla_x U(X)$ is changed to its opposite. We have complete freedom for the choice of what happens to the part of $Y$ which is orthogonal to $\nabla_x U(X)$, as long as it stays orthogonal.

In particular the condition would be satisfied with $p_{x,y} = \delta_{-y}$, namely $Y$ jumps to $-Y$. But then $X_t$ would be forever trapped on the same line, and the process would not be ergodic. This could be fixed by adding uniform jumps of $Y$ at a constant rate, which leaves $\nu$ invariant, but then the trajectories would still seem rather inefficient from a mixing point of view.

The choice $p_{x,y} = \delta_{y^*}$ seems more natural.  From a practical point of view, it means that the process has more inertia than if we had chosen the kernel $Y \leftarrow -Y$. However to sample the process, it is then necessary to estimate $\n_x U$ at each jump time; but this is still less than for a diffusion. 

It is not exactly clear whether it is ergodic: if $U$ is infinite outside an ellipse and vanishes inside, then there is no randomness, the process is a deterministic billiard and so even if its initial speed is uniformly distributed, it will stay in an area defined by some caustic. If we add uniform jumps on the velocity at constant rate $r>0$, then obviously the process may reach any open set of $\R^d\times \mathbb S^{d-1}$ with positive probability.

All these considerations lead to the definition \eqref{EqGenerateurDimD}

\subsection*{Outine of the paper}

In the rest of the paper are proven the above results. Section \ref{SectionTemperatureConstante} is concerned with the low-temperature regime for a double-well potential; Theorem \ref{ThmCVexpoDimUn} is proven, which introduces a discussion on the time scale of total variation and Wasserstein convergence to equilibrium. The proof of Theorem \ref{TheoremNS}, in inhomogeneous setting, is addressed in Section \ref{SectionInhomogene}. Section~\ref{SectionNonMinimal} investigates the non-minimal rate case, namely the case $\lambda(x,y) > \po y U'(x)\pf_+$. Finally, the multi-dimensional case is addressed in Section \ref{SecDimd}, where Theorems \ref{ThmDimDConstant} and \ref{ThmDimDRecuit} are proven.

\bigskip

 \textbf{Acknowledgements.} The author would like to thank Carl-Erik Gauthier for his help in improving this paper.

\section{Escape time at low temperature}\label{SectionTemperatureConstante}

In this section the dimension is 1, the temperature $\varepsilon>0$ is fixed through time, $(X_t^\varepsilon,Y_t^\varepsilon)$ is the velocity jump process with generator $L$ defined by \eqref{EqGenerateur}; the double-well potential $U$ has three local extrema $x_0<x_1<x_2$ such as described in the previous section in the settings of Theorem \ref{ThmCVexpoDimUn}. Suppose $(X_0^\varepsilon,Y_0^\varepsilon) = (x_0,-1)$ and recall that
\begin{eqnarray*}
 \tau & = & \inf \left\{t>0, X_t^\varepsilon =x_1\right\}.
\end{eqnarray*}
We start the proof of Theorem \ref{ThmCVexpoDimUn} with the following Lemma :
\begin{lem}\label{LemRacineEpsi}
 for $\delta >0$ small enough,
 \begin{eqnarray*}
 \int_0^\delta\frac{t}{\varepsilon}  \po -U'(x_0-t)\pf e^{-\frac{U(x_0-t)-U(x_0)}\varepsilon} dt & = & \sqrt{\frac{\pi \varepsilon}{2 U''(x_0)}} \po 1 + \underset{\varepsilon \rightarrow 0}o (1)\pf.
\end{eqnarray*}
 \end{lem}
 \begin{proof}
 \begin{eqnarray*}
 \int_0^\delta\frac{t}{\varepsilon}  \po -U'(x_0-t)\pf e^{-\frac{U(x_0-t)-U(x_0)}\varepsilon} dt & = & \int_0^{\frac{\delta}{\sqrt \varepsilon}} s  \po -U'(x_0-\sqrt \varepsilon s)\pf e^{-\frac{U(x_0-\sqrt \varepsilon s)-U(x_0)}\varepsilon} ds\\
 & =& \sqrt \varepsilon \int_0^{\infty} f_\varepsilon(s)  ds
\end{eqnarray*}
with
\[f_\varepsilon(s) =  \frac{s  \po -U'(x_0-\sqrt \varepsilon s)\pf}{\sqrt \varepsilon} e^{-\frac{U(x_0-\sqrt \varepsilon s)-U(x_0)}\varepsilon} \mathbb 1_{s\leq \delta/\sqrt \varepsilon}.\]
On the one hand
\[f_\varepsilon(s) \underset{\varepsilon \rightarrow 0_+}\longrightarrow s^2  U''(x_0)   e^{-\frac{U''(x_0)}2 s^2 }\]
and on the other hand, writing $M = \underset{t\in[0,\delta]}\sup | U^{(3)}(t)|$, 
\[|f_\varepsilon(s) |\leq s^2  \po U''(x_0) + \frac{\delta M}{2}\pf   e^{-\frac{U''(x_0)}2 s^2 + \frac{\delta s^2 M}{6}}.\]
As $M$ decreases with $\delta$, for $\delta$ small enough, $\delta M < 3U''(x_0)$.
Thus by the dominated convergence theorem
\begin{eqnarray*}
 \int_0^{\infty} f_\varepsilon(s)  ds & \underset{\varepsilon \rightarrow 0}\longrightarrow &\int_0^\infty s^2  U''(x_0)   e^{-\frac{U''(x_0)}2 s^2 } ds\\
 & = &  \sqrt{\frac{\pi}{2 U''(x_0)}}.
\end{eqnarray*}
\end{proof}

\begin{proof}[Proof of Theorem \ref{ThmCVexpoDimUn}]
The strategy of this proof is quite simple: starting at $(x_0,-1)$, the process starts to climb to the left. After a time with bounded expectation (since the potential is increasing and convex along the trajectory), it turns back, reaches $(x_0,1)$ and start to climb toward $x_1$. If it reaches $x_1$ then this attempt to escape is a success. If not, it goes back to $(x_0,-1)$, and starts anew. Since the duration of the last (and successful) attempt is negligible with respect to the expected duration of a failed attempt (of order $\sqrt{\varepsilon}$ according to Lemma \ref{LemRacineEpsi}) times the expected number of failed attempts, conclusion follows from the usual convergence of geometrical laws to the exponential one.

More precisely let $\theta_0 = 0$, and for all $k\in\mathbb N$ set
 \begin{eqnarray*}
  \tilde \theta_k & = & \min\po \inf \ao t > \theta_k,\ \po X_t,Y_t\pf = (x_0,1)\af,\ \tau\pf,\\
  \theta_{k+1} & = &  \min\po\inf \ao t > \tilde \theta_k,\ \po X_t,Y_t\pf = (x_0,-1)\af,\ \tau\pf.
 \end{eqnarray*}
 We also define  $N = \max\ao k,\ \theta_k < \tau\af$, so that
\begin{eqnarray*}
 \tau & = & \tau - \theta_N +  \sum_{k=1}^N (\theta_{k} - \theta_{k-1}) .
\end{eqnarray*}
Conditionally on $N\geq k$, $\theta_{k} - \theta_{k-1}$ is the length of a failed attempt to reach $x_1$; it has the same law as $\theta_1$, given $N\geq 1$, and is independent from $N$. In particular
\begin{eqnarray}\label{EquationLemme}
 \E \co \tau \cf &  = & \E\co \tau - \theta_N \cf  + \E\co N\cf  \E \co\theta_1\ |\ N\geq1\cf,
\end{eqnarray}
and $N$ is a geometric variable with parameter
\begin{eqnarray*}
 q_\varepsilon & = & \mathbb P \po \varepsilon E > U(x_1)-U(x_0) \pf\\
 & = & e^{- \frac{U(x_1)-U(x_0)}{\varepsilon}}
\end{eqnarray*}
(and so with expectation  $\E\co N\cf = q_\varepsilon^{-1}$). 
Now we decompose
\begin{eqnarray}\label{EquationLemme2}
 \E \co\theta_1\ |\ N\geq1\cf &  =  &2 \E \co\frac{\tilde \theta_0}2\cf + 2  \E \co\left.\frac{\theta_1 - \tilde \theta_0}2\ \right|\ N\geq1\cf.
\end{eqnarray}
Notice that $\frac{\tilde \theta_0}2$  is the jump time starting from $(x_0,-1)$, which is independent from $N$ and whose law has density $\frac1\varepsilon \po -U'(x_0-t)\pf e^{-\frac{U(x_0-t)-U(x_0)}\varepsilon}$. 
If $\varepsilon < \frac12$, then for any $\delta >0$
\begin{eqnarray*}
\E \co \frac{\tilde \theta_0}2 \mathbbm 1_{\frac{\tilde \theta_0}2>\delta}\cf & = & 
 \int_\delta^\infty \frac{t}{\varepsilon}  \po -U'(x_0-t)\pf e^{-\frac{U(x_0-t)-U(x_0)}\varepsilon} dt\\
 & \leq & \int_\delta^\infty \frac{t}{\varepsilon}  \po -U'(x_0-t)\pf e^{-\frac{U(x_0-t)-U(x_0)}{2\varepsilon} - \frac{U(x_0-\delta)-U(x_0)}{2\varepsilon}} dt\\
 & \leq & \frac{ e^{ - \frac{U(x_0-\delta)-U(x_0)}{2\varepsilon}}}{\varepsilon} \int_\delta^\infty  t \po -U'(x_0-t)\pf e^{-\po U(x_0-t)-U(x_0)\pf } dt\\
 & = & \underset{\varepsilon \rightarrow 0}o (\sqrt{\varepsilon}).
\end{eqnarray*}
On the other hand thanks to Lemma \ref{LemRacineEpsi}, for $\delta$ small enough,
\begin{eqnarray*}
\E \co \frac{\tilde \theta_0}2 \mathbbm 1_{\frac{\tilde \theta_0}2<\delta}\cf & = & \sqrt{\frac{\pi \varepsilon}{ 2U''(x_0)}} \po 1 + \underset{\varepsilon \rightarrow 0}o (1)\pf.
\end{eqnarray*}
If $\delta < x_1-x_0$, Lemma \ref{LemRacineEpsi} also applies to $\frac{\theta_1 - \tilde \theta_0}2$, and since $\mathbb P(N\geq 1)$ goes to 1,
\begin{eqnarray*}
\E \co \left. \frac{\theta_1 - \tilde \theta_0}2 \mathbbm 1_{\frac{\theta_1 - \tilde \theta_0}2<\delta}\right| N\geq1\cf & = & \sqrt{\frac{\pi \varepsilon}{ 2U''(x_0)}} \po 1 + \underset{\varepsilon \rightarrow 0}o (1)\pf.
\end{eqnarray*}
Furthermore
\begin{eqnarray*}
 \E \co\left.\frac{\theta_1 - \tilde \theta_0}2 \mathbbm 1_{\frac{\theta_1 - \tilde \theta_0}2\geq\delta}\ \right|\ N\geq1\cf & \leq & (x_1-x_0)\frac{\mathbb P \po \frac{\theta_1 - \tilde \theta_0}2\geq\delta\pf}{\mathbb P(N\geq 1)} \\
 & = & \underset{\varepsilon \rightarrow 0}o (\sqrt{\varepsilon}).
\end{eqnarray*}
By combining the estimates above with \eqref{EquationLemme2} we get
\begin{eqnarray*}
  \E \co\theta_1\ |\ N\geq1\cf & = & \sqrt{\frac{8 \pi \varepsilon}{ U''(x_0)}} \po 1 + \underset{\varepsilon \rightarrow 0}o (1)\pf.
\end{eqnarray*}
Besides, note that $ \frac{ \tilde\theta_N - \theta_N}2$ has the same law as $\frac{\tilde \theta_0}2$, so that 
\[\E\co \tau - \theta_N \cf  = 2 \E\co \frac{ \tilde\theta_N- \theta_N}2 \cf + x_1 - x_0 = \underset{\varepsilon \rightarrow 0}O (1) = \underset{\varepsilon \rightarrow 0}o \po \sqrt{\varepsilon}q_\varepsilon^{-1}\pf, \]
and finally Equality \eqref{EquationLemme} becomes
\[ \E \co \tau\cf =   \sqrt{\frac{8\pi \varepsilon}{ U''(x_0)}}  e^{\frac{U(x_1)-U(x_0)}{\varepsilon}} \po 1 + \underset{\varepsilon \rightarrow 0}o (1)\pf.\]

Now as far as the second assertion of the theorem is concerned,
\begin{eqnarray*}
 \lefteqn{\mathbb P\po \tau \geq t \E\co \tau \cf \pf}\\
  & = &   \mathbb P\po  \tau - \theta_N +  \sum_{k=1}^N (\theta_{k} - \theta_{k-1}) \geq t \E\co \tau \cf \pf\\
 & = &   \mathbb P\po \frac{\E\co N \cf \E\co \theta_1 \ |\ N\geq1\cf}{\E\co \tau \cf} \frac{N}{\E\co N \cf }\po\frac{ \tau - \theta_N}{N\E\co \theta_1 \ |\ N\geq1\cf} +  \sum_{k=1}^N \frac{\theta_{k} - \theta_{k-1}}{N \E\co \theta_1 \ |\ N\geq1\cf} \pf \geq t  \pf\\
\end{eqnarray*}
First, $\frac{N}{\E\co N \cf }$ converges in law to an exponential variable of parameter 1. Second, if the times of jump of the process are defined by the same sequence $\po E_k\pf_{k\geq 0}$ of exponential variables (according to Equality \eqref{DefiTavecExpo}) for all $\varepsilon$, then the processes at different temperatures are all defined at once on the same probability space, and in this case 
\[N = \min\{k\geq0, \varepsilon E_{2k} > U(x_1)-U(x_0)\}\]
almost surely goes to infinity. Thanks to the law of large numbers,
\[\frac{ \tau - \theta_N}{N\E\co \theta_1 \ |\ N\geq1\cf} +  \sum_{k=1}^N \frac{\theta_{k} - \theta_{k-1}}{N \E\co \theta_1 \ |\ N\geq1\cf} \underset{\varepsilon\rightarrow 0}\longrightarrow 1\hspace{15pt}a.s.\]

At last $ \frac{\E\co N \cf \E\co \theta_1 \ |\ N\geq1\cf}{\E\co \tau \cf} \underset{\varepsilon\rightarrow 0}\longrightarrow 1$. As a conclusion $\frac{\tau}{\E\co \tau\cf}$ converges to an exponential law.
\end{proof}

This result yields a lower bound on the total variation distance between
\[\nu_t = \mathcal L\po X_t^\varepsilon,Y_t^\varepsilon\ |\ X_0^\varepsilon=x_0,\ Y_0^\varepsilon=-1\pf\]
and the invariant measure $\mu_\varepsilon$. We recall that the total variation distance between two laws $\nu,\tilde \nu$ on a topological space $W$ is defined by
\[\|\nu-\tilde \nu\|_{TV} = \underset{A\in \mathcal B(W)}\sup |\nu(A)-\tilde \nu(A)|\]
where $\mathcal B(W)$ is the Borel $\sigma$-algebra on $W$.
\begin{cor}
In the setting of Theorem \ref{ThmCVexpoDimUn}, assume furthermore $U(x_2) < U(x_0)$. For any $t>0$, writing $t_\varepsilon = t\E\co \tau \cf$,
 \begin{eqnarray*}
  \underset{\varepsilon\rightarrow 0}\liminf\ \|\nu_{t_\varepsilon}-\mu_{\varepsilon}\|_{TV} & \geq & e^{-t}. 
 \end{eqnarray*}
\end{cor}
\begin{proof}
 The assumption $U(x_2) < U(x_0)$ implies that the first marginal of $\mu_\varepsilon$ concentrates near the unique global minimum of $U$, $x_2$. In particular
 \[\mu_\varepsilon\co  \po x_1,\infty\pf\times\{\pm1\}\cf \underset{\varepsilon\rightarrow 0}\longrightarrow 1.\]
However
 \begin{eqnarray*}
  \nu_t \co \po x_1,\infty\pf\times\{\pm1\}\cf & \leq & \mathbb P\po \tau \leq t\pf.
 \end{eqnarray*}
As a consequence
\begin{eqnarray*}
\|\nu_{t_\varepsilon}-\mu_{\varepsilon}\|_{TV} & \geq & \po \mu_{\varepsilon} - \nu_{t_\varepsilon} \pf\co  \po x_1,\infty\pf\times\{\pm1\}\cf  \\
 & \geq & \mu_{\varepsilon} \co  \po x_1,\infty\pf\times\{\pm1\}\cf - \mathbb P\po \tau \leq t \E\co \tau\cf\pf\\
 & \underset{\varepsilon\rightarrow 0}\longrightarrow& e^{-t}.
\end{eqnarray*}
\end{proof}

We will now give a more precise result concerning the first marginal and the Wasserstein distance
\[\mathcal W \po \nu_1,\nu_2\pf = \underset{\pi}\sup \E_\pi\co |Z_1-Z_2| \cf,\]
where the infimum is taken over all probability measures $\pi$ on $\R^2$ whose first (resp. second) marginal is $\nu_1$ (resp. $\nu_2$). We start with the following lemma (recall that we fixed $(X_0^\varepsilon,Y_0^\varepsilon) = (x_0,-1)$) :
\begin{lem}\label{LemEspBorne}
  Let $t>0$. Then
  \[\underset{\varepsilon\rightarrow 0}\limsup \E \co (X_{t_{\varepsilon}}^\varepsilon)^2 \cf < \infty.\]
\end{lem}
\begin{proof}
Let $M>U(x_1)$, and for all $k\geq0$ let
\[s_k = \inf \{ s>0, \ U(X_s^\varepsilon) > M+k\}.\]
From Theorem  \ref{TheoremLowerBound} (stated and proved in the next section) we get that for all $k\geq0$, there exist $ \Gamma,\varepsilon_0>0$ such that for all $\varepsilon\leq \varepsilon_0$ and all $t>0$,
\begin{eqnarray*}
 \mathbb P \po s_k \leq t_\varepsilon \pf &\leq& 1- \exp\po - \Gamma \po e^{-\frac{M+k}{\varepsilon}}+ t_\varepsilon \frac{e^{-\frac{M+k - U(x_2)}{\varepsilon}}}{\sqrt{\varepsilon}} \pf\pf\\
 & \leq & 1- \exp\po - \tilde \Gamma  e^{-\frac{M+k}{\varepsilon_0}} \pf,
\end{eqnarray*}
where $\tilde \Gamma$ depends on $t$ but not on $\varepsilon$. In fact, since we consider here escape times from nested intervals, and from the convexity of $U$ outside of a compact, $\Gamma$ (and hence $\tilde \Gamma$) and $\varepsilon_0$ can be chosen uniformly over $k$ (see the remark at the end of the proof of Theorem \ref{TheoremLowerBound}). Noticing that $U^{-1}(v) = \sup\{|x|,\ U(x)\leq v\}$ is sub-linear, we write
\begin{eqnarray*}
 \E\co (X_{t_\varepsilon}^\varepsilon)^2  \cf & \leq & \po U^{-1}(M)\pf^2 +  \sum_{k\geq 0} \po U^{-1}(M+k+1)\pf^2\mathbb P\po s_k \leq t_\varepsilon\pf\\
 & \leq & \po U^{-1}(M)\pf^2 +  \sum_{k\geq 0} \po U^{-1}(M+k+1)\pf^2\po1- \exp\po - \tilde \Gamma  e^{-\frac{M+k}{\varepsilon_0}} \pf\pf\\
 & \leq & \po U^{-1}(M)\pf^2 +  \tilde \Gamma \sum_{k\geq 0} \po U^{-1}(M+k+1)\pf^2 e^{-\frac{M+k}{\varepsilon_0}} 
\end{eqnarray*}
which is finite.

 


\end{proof}

Let $m_t = e^{-t}\delta_{x_0} + (1-e^{-t})\delta_{x_2}$ and $h_t$ be the first marginal of $\nu_t$, namely
 \[h_t = \mathcal L\po X_t^\varepsilon\ |\ X_0^\varepsilon=x_0,\ Y_0^\varepsilon=-1\pf.\]

\begin{thm}\label{ThmWasser}
In the setting of Theorem \ref{ThmCVexpoDimUn}, assume moreover $U(x_2) < U(x_0)$. Then for any $t>0$, writing $t_\varepsilon = t\E\co \tau \cf$,
 \[\mathcal W\po h_{t_\varepsilon},m_t\pf \underset{\varepsilon\rightarrow 0}\longrightarrow 0.\]
\end{thm}
\begin{proof}
Consider $(X_t^\varepsilon,Y_t^\varepsilon)_{t\geq0}$ a trajectory of the process, from which we will define a variable $W$ of law $m_t$. In the first instance assume $p_\varepsilon := \mathbb P\po \tau\geq t_\varepsilon \pf \leq e^{-t}$. If $\tau\geq t_\varepsilon$, set $U=1$. Else set $U=1$ with probability $\frac{e^{-t}-p_\varepsilon}{1-p_\varepsilon}$, and else set $U=0$. Similarly in the case where $p_\varepsilon \geq e^{-t}$, if $\tau\leq t_\varepsilon$, set $U=0$, and else set $U=0$ with probability $\frac{p_\varepsilon-e^{-t}}{p_\varepsilon}$, and else set $U=1$.

Either ways, $U$ is a Bernoulli variable with parameter $e^{-t}$ such that
 \[\mathbb P \po U = \mathbbm 1_{\tau\geq t_\varepsilon}\pf = 1 - |e^{-t}-p_\varepsilon|\underset{\varepsilon\rightarrow 0}\longrightarrow 1.\]
 We naturally set $W = U x_0 + (1-U) x_2$. The Cauchy-Schwarz inequality
 \[\E\co |X_{t_\varepsilon}^\varepsilon - W|\mathbbm 1_{U \neq \mathbb 1_{\tau\geq t_\varepsilon}} \cf  \leq \sqrt {\E\co |X_{t_\varepsilon}^\varepsilon - W|^2 \cf \mathbb P\po U  \neq \mathbbm 1_{\tau\geq t_\varepsilon}\pf},\]
 together with Lemma \ref{LemEspBorne}, yields
 \begin{eqnarray*}
  \E\co |X_{t_\varepsilon}^\varepsilon - W|\cf & =& \E\co |X_{t_\varepsilon}^\varepsilon - W| \mathbb 1_{ U = \mathbbm 1_{\tau\geq t_\varepsilon}}\cf +\underset{\varepsilon\rightarrow 0} o(1).
 \end{eqnarray*}
We write
\begin{eqnarray*}
 \tau^{(2)} & = & \inf \ao t > \tau,\ X_t^\varepsilon<x_1\af,
\end{eqnarray*}
and remark that $\tau^{(2)}-\tau$, which is independent from $\tau$, is the first hitting time of $(x_1,-1)$ starting from $(x_2,1)$: from Theorem \ref{ThmCVexpoDimUn}, and since $U(x_2)<U(x_1)$, we see that $\E\co \tau \cf$ is negligible with respect to $\E\co\tau^{(2)}-\tau\cf$, that $\frac{\tau^{(2)}-\tau}{\E\co\tau^{(2)}-\tau\cf}$ converges in law to an exponential variable and so that $\frac{\tau^{(2)}-\tau}{\E\co\tau\cf}$ diverges in probability to infinity. In particular
\begin{eqnarray*}
\mathbb P\po \tau^{(2)} \geq t_\varepsilon\pf & \geq  & \mathbb P\po \tau^{(2)} - \tau \geq t_\varepsilon \pf \\
  & \underset{\varepsilon \rightarrow 0}\longrightarrow & 1 .
\end{eqnarray*}
Using again Lemma \ref{LemEspBorne}, we get
 \begin{eqnarray*}\E\co |X_{t_\varepsilon}^\varepsilon - W|\cf
& = & \E\co |X_{t_\varepsilon}^\varepsilon - W|\ \mathbb 1_{ U = \mathbb 1_{\tau\geq t_\varepsilon},\ \tau^{(2)} > t_\varepsilon}\cf +\underset{\varepsilon\rightarrow 0} o(1)\\
& = & \E\co |X_{t_\varepsilon}^\varepsilon - x_0|\ \mathbb 1_{ t_\varepsilon < \tau}\cf + \E\co |X_{t_\varepsilon}^\varepsilon - x_2|\ \mathbb 1_{\tau < t_\varepsilon < \tau^{(2)}}\cf +\underset{\varepsilon\rightarrow 0} o(1).
 \end{eqnarray*}
Both expectations are treated the same way, let us focus on the first one. Suppose $t_\varepsilon < \tau$ and let
\[t' = \sup\{ s<t_\varepsilon, \ X_s^\varepsilon=x_0\},\hspace{20pt}t'' = \min\po \inf\{ s>t_\varepsilon, \ X_s^\varepsilon=x_0\},\tau\pf,\]
and $ I = t''-t'$. Note that $|X_{t_\varepsilon}^\varepsilon-x_0|<I$ and  for $0\leq a \leq b$ the following events are included:
\[\left\{ a \leq I < b\right\} \subset A,\]
where
\[A = \left \{\exists s_1,s_2\in[t_\varepsilon-b,t_\varepsilon+b],\ s_1<s_2,\ X_{s_1}^\varepsilon=x_0,\ |X_{s_2}^\varepsilon-x_0| = \frac{a}{2}\right\}.\]
Since the process starts anew when it reaches $x_0$,
\begin{eqnarray*}
 \mathbb P(A) & \leq & \mathbb P_{X_0 = x_0}\po \exists s < 2b,\ |X_s^\varepsilon-x_0| = \frac{a}{2}\pf\\
 & \leq & \mathbb P\po \theta(a) \leq 2b\pf,\end{eqnarray*}
 where 
 \begin{eqnarray*}
 \theta(a) & = & \inf\{s>0,\ U(X_s^\varepsilon) > U(x_0) + \delta(a)\}\\
 \text{and}\hspace{20pt}\delta(a) & =& \min\po U\po x_0 - \frac{a}{2}\pf, U\po x_0 + \frac{a}{2}\pf  \pf - U(x_0)\text{ if }x_0 + \frac{a}{2}\leq x_1\\
 & = &  U\po x_0 - \frac{a}{2}\pf-U(x_0)\text{ else}.
\end{eqnarray*}
In other words, $\delta(a)$ is the minimal energy barrier to overcome to be at distance at least $\frac{a}{2}$ from $x_0$ (when the process is still in the catchment area of $x_0$). If $a\leq I <b$, such a distance has been attained, and it has been so in a prescribed time window of length at most $2b$. As in the proof of Lemma \ref{LemEspBorne} we use Theorem \ref{TheoremLowerBound} to control the probability to reach a given energy level in a given time: for all $k\geq0$ and $\varepsilon < \varepsilon_0$,
\begin{eqnarray*}
 \mathbb P\po \theta(a) \leq 2b \pf & \leq & 1- \exp\po - \Gamma \po e^{-\frac{\delta(a)}{\varepsilon}}+ 2b\frac{e^{-\frac{\delta(a)}{\varepsilon}}}{\sqrt{\varepsilon}} \pf\pf\\
 & \leq & \Gamma \po 1 + 2b\varepsilon^{- \frac12} \pf e^{-\frac{\delta(a)}{\varepsilon}}.
\end{eqnarray*}
By combining the estimates above, when $\varepsilon<1$,
\begin{eqnarray}\label{EqSecondeSomme}
& &  \E\co |X_{t_\varepsilon} - x_0|\ \mathbb 1_{ t_\varepsilon < \tau}\cf \notag\\
& \leq & \E\co I\ \mathbb 1_{ t_\varepsilon < \tau}\cf \notag\\
 & \leq & \sum_{0\leq k\leq \left \lceil\varepsilon^{-\frac12}\right \rceil} \sqrt \varepsilon (k+1) \mathbb P\po k \sqrt{\varepsilon} \leq I < (k+1)\sqrt{\varepsilon}\pf + \sum_{k\geq 1} (k+1) \mathbb P\po k  \leq I < (k+1)\pf \notag\\
 &\leq  & 4\Gamma \sqrt \varepsilon  \sum_{0\leq k\leq \left \lceil\varepsilon^{-\frac12}\right \rceil} (k+1)^2e^{-\frac{\delta\po k\sqrt \varepsilon/2\pf}{\varepsilon}} + 4\Gamma\sum_{k\geq 1}(k+1)^2\varepsilon^{- \frac12}  e^{-\frac{\delta(k)}{\varepsilon}}.
\end{eqnarray}
From the convexity of $U$, $\delta(k)$ grows faster than linearly, so that, if $\varepsilon_0$ is such that
\[\varepsilon \mapsto \varepsilon^{- \frac12}  e^{-\frac{\delta(1)}{\varepsilon}}\]
is non-increasing on $(0,\varepsilon_0)$, for all $\varepsilon < \varepsilon_0$,
\[\sum_{k\geq 1}(k+1)^2\varepsilon^{- \frac12}  e^{-\frac{\delta(k)}{\varepsilon}} < \sum_{k\geq 1}(k+1)^2\varepsilon_0^{- \frac12}  e^{-\frac{\delta(k)}{\varepsilon_0}} < \infty.\]
Thus, by the dominated convergence Theorem, the second sum in \eqref{EqSecondeSomme} vanishes as $\varepsilon$ goes to 0. As far as the first sum is concerned, note that, since $x_0$ is a non-degenerated minimum of $U$, there exist $\eta>0$ such that $\delta(s) \geq \eta s^2$ for all $s\in[0,1]$. This implies
\[\sum_{0\leq k\leq \left \lceil\varepsilon^{-\frac12}\right \rceil} (k+1)^2e^{-\frac{\delta\po k\sqrt \varepsilon/2\pf}{\varepsilon}} \leq \sum_{k\geq 0} (k+1)^2e^{-\eta k^2} < \infty.\]
By inserting this estimate in \eqref{EqSecondeSomme},   $\E\co |X_{t_\varepsilon}^\varepsilon - x_0|\ \mathbb 1_{ t_\varepsilon < \tau}\cf$ goes to 0; $ \E\co |X_{t_\varepsilon}^\varepsilon - x_2|\ \mathbb 1_{\tau < t_\varepsilon < \tau^{(2)}}\cf$ is treated the same way, and we conclude that $\E\co |X_{t_\varepsilon}^\varepsilon - W|\cf$ goes to 0. For $ \varepsilon$ small enough one can thus find a coupling $(X,W)$ with marginal laws $\mathcal L\po X_{t_\varepsilon}^\varepsilon\pf $ and $e^{-t}\delta_{x_0} + (1-e^{-t})\delta_{x_2}$ such that $\E\co |X-W|\cf$ is arbitrarily small.
\end{proof}

In particular since $\mu^1_\varepsilon = e^{-\frac1\varepsilon U(x)} dx$, the first marginal of the invariant measure $\mu_\varepsilon$, converges to the Dirac mass on $x_2$, and
\[\mathcal W \po e^{-t}\delta_{x_0} + (1-e^{-t})\delta_{x_2}, \delta_{x_2}\pf = e^{-t}|x_2-x_0|,\]
we get
\begin{cor}
In the setting of Theorem \ref{ThmWasser},
\[\mathcal W \po h_{t_\varepsilon},\mu^1_\varepsilon\pf \underset{\varepsilon\rightarrow0}\longrightarrow e^{-t}|x_2-x_0|.\]
\end{cor}

\section{NS conditions for the cooling schedule}\label{SectionInhomogene}

We now turn to the study  of the inhomogeneous velocity jump process $\po X_t^{\varepsilon_t},Y_t^{\varepsilon_t}\pf_{t\geq 0}$ on $\R\times \{-1,+1\}$ with generator
\begin{eqnarray*}
 L_t f(x,y) & = & y f'(x,y) + \frac1{\varepsilon_t}\po y U'(x)\pf_+\po f(x,-y) - f(x,y)\pf
\end{eqnarray*}
when Assumption \ref{HypoUepsi} holds, which is supposed in the whole Section \ref{SectionInhomogene}. To lighten the notations, in this section we drop the $\varepsilon_t$ exponent and only call the process $(X_t,Y_t)_{t\geq 0}$.

It can be explicitly constructed in the same manner as the homogeneous process of Section~\ref{SectionDefRTP}, except that the definition \eqref{DefiTavecExpo} of the jump times is now replaced by
\begin{eqnarray*}
 T_{i+1} & = & \inf\left\{t > T_i,\ \int_{T_i}^t \frac{\po Y_{T_i}U'\po X_{T_i} + Y_{T_i}(s-T_i)\pf\pf_+}{\varepsilon_s} ds\geq E_i\right \}.
\end{eqnarray*}
In a finite time interval $[0,T]$, $\varepsilon_t$ is bounded by $\varepsilon_T$ and thus there cannot be an infinite number of jumps in a finite time. In particular the sequence of $(T_i)_{i\geq 0}$ goes to infinity and the process is well defined at all times.

Our method to prove Theorem \ref{TheoremNS} follows the work of Hajek (\cite{Hajek}) on simulated annealing on a discrete space. We consider a smooth Morse potential $U$ on $\R$ with a finite number of local extrema, and convex at infinity. If $x$ is a minimum of depth $E$, denote by $C_x$ the set of all points which are reachable from $x$ at height strictly less than $E$. We call $C_x$ the cusp of $x$ (see the grey area of Figure \ref{FigureProfondeur}). 

More generally we call cusp an interval  $C=(z_l,z_r)$ with $U(z_l) = U(z_r)$ and for all $x\in C$, $U(x)<U(z_l)$. The depth $d$ of $C$ is defined as
\[d = U(z_l) - \min\{U(x),\ x\in C\}.\]
We denote by $B=\left\{z\in C, U(z) = \min\{U(x),\ x\in C\}\right\}$ the bottom of the cusp. Note that a minimum $x$ is always in the bottom of $C_x$, and that conversely if $z$ is in the bottom of $C_x$ then $C_x=C_z$. Obviously the depth of $C_x$ equals the depth of $x$.

\bigskip

We want to bound the time the process spends in a cusp $C$, depending on the depth $d$ of the latter. Nevertheless it is impossible to do so if we only assume the initial position is in $C$: we should put aside the cases where the process starts near the boundary of $C$, with a velocity directed toward the exit.

We denote by $\mathcal N$ the set of local extrema of $U$ in $C$, and
\[u = U(z_l) - \max\{ U(z),\ z\in\mathcal N\}.\]
The maximum over $\mathcal N$ is necessarily attained on a local maximum, except when there is no local maximum between  $z_l$ and $z_r$, in which case $\mathcal N$ is a single point, which is a local minimum, and $u=d$. We define
\[x_l = \inf\{ x\in C, U(z_l)-U(x) \geq u\}\hspace{15pt} x_r = \sup\{ x\in C, U(z_r)-U(x) \geq u\},\]
and
\[A_C = (z_l,x_l)\times\{-1\} \cup (x_r,z_r)\times\{1\} \subset \R\times \{\pm 1\}.\]
The process can only leave $C$ from $A_C$, and can only enter it through $\bar A_C= \po C\times \{\pm 1\}\pf\smallsetminus A_C$. We will note $A_C(t)$ (resp. $\bar A_C(t)$) the event $(X_{t},Y_t)\in A_C$ (resp. $\bar A_C$).

\bigskip

Let $t_0 >0$ and
\begin{eqnarray*}
 \tau_C & = & \inf\{t>t_0,\ X_t\notin C\}.
\end{eqnarray*}
In order to prove Theorem \ref{TheoremNS}, we will establish the two following intermediate results:

\begin{thm}\label{TheoremsortieCoupe}
 Let $C$ be a cusp of depth $d$. There exist $\varepsilon_0,\ c > 0$ such that for all time $t_0>0$, any $z\in\{z_l,z_r\}$ and all cooling schedule $t\mapsto \varepsilon_t$, if $\bar A_C(t_0)$ holds, $\varepsilon_{t_0} \leq \varepsilon_0$ and 
 \[\int_{t_0}^{\infty} \po \varepsilon_s\pf^{-\frac12} e^{-\frac{d}{\varepsilon_s}}ds = \infty,\]
 then 
 \begin{eqnarray}
  \E\co \int_{t_0}^{\tau_C} \po \varepsilon_s\pf^{-\frac12} e^{-\frac{d}{\varepsilon_s}}ds\cf & \leq & c, \label{PartieUnTheoremeSortie}\\
  \mathbb P\po X_{\tau_C} = z \pf & \geq &  \frac1c.\label{PartieDeuxTheoremeSortie}
 \end{eqnarray}
\end{thm}
\begin{thm}\label{TheoremLowerBound}
Let $C$ be a cusp of depth $d$. There exist $\Gamma,\varepsilon_{0}>0$   such that for all times $t_0>0$, $r\geq t_0$ and all cooling schedule $t\mapsto \varepsilon_t$, if $\bar A_C(t_0)$ holds and $\varepsilon_{t_0} \leq \varepsilon_0$ then
 \[\mathbb P \po \tau_C \geq r\pf \geq \exp\po - \Gamma \po \frac{e^{-\frac{d}{\varepsilon_{t_0}}}}{\sqrt{\varepsilon_{t_0}}}+ \int_{t_0}^{r} \frac{e^{-\frac{d}{\varepsilon_{u}}}}{\sqrt{\varepsilon_{u}}} du \pf\pf.\]
\end{thm}

Note that contrary to Theorem \ref{TheoremNS}, these intermediate results do not require the temperature to go to zero. In particular they hold for constant $\varepsilon$.

\subsection{Theorem \ref{TheoremsortieCoupe}}

The proof is based on an induction over the number of local minima that are contained in the cusp C.  Thus we start with a cusp that contains only one minimum, and so no maximum.

\subsubsection{A simple cusp}\label{SubsectionCuspSimple}

We fix throughout this section a cusp $C$ of depth $d$ with only one local minimum $x_0$ of $U$. Note that in this case  $u=d$, and $x_l=x_r=x_0$. We will prove Theorem \ref{TheoremsortieCoupe} in this situation, which is pretty similar to the settings of Section \ref{SectionTemperatureConstante}. Let $\tau_0 = t_0$ and 
\begin{eqnarray*}
 \tau_{i+1} &= & \min \po \inf\{t>\tau_i,\ X_t=x_0\},\ \tau_C\pf.
\end{eqnarray*}
We first prove:
\begin{lem}\label{LemControleRacine}
 There exists $c_1$ and $\varepsilon_0>0$ such that for all $\varepsilon_{t_0} \leq \varepsilon_0$, for all $i\geq 1$,
 \[\E\co \left. \tau_{i+1}-\tau_i \right| \bar A_C(\tau_i), \tau_i\cf \leq c_1 \sqrt{\varepsilon_{\tau_i}}. \]
\end{lem}
\begin{proof}
At time $\tau_i$, $i\geq1$, given $ \bar A_C(\tau_i)$, the position $X_{\tau_i}$ is $x_0$ and the velocity $Y_{\tau_i}$ is either 1 or -1. Writing
 \begin{eqnarray*}
  \E\co \left. \tau_{i+1}-\tau_i \right| \bar A_C(\tau_i), \tau_i\cf & =& \E\co \left. \E\co \left. \tau_{i+1}-\tau_i \right| \bar A_C(\tau_i), \tau_i, Y_{\tau_i}\cf \right| \bar A_C(\tau_i), \tau_i\cf, 
 \end{eqnarray*}
both cases $Y_{\tau_i} = -1$ or $1$ are treated the same way. For instance
 \begin{eqnarray*}
   \E\co \left. \tau_{i+1}-\tau_i \right|  \tau_i, \po X_{\tau_i},Y_{\tau_i}\pf = (x_0,-1)\cf 
  & =& \mathbb P\po \tau_{i+1} = \tau_C |  \tau_i, \po X_{\tau_i},Y_{\tau_i}\pf = (x_0,-1)\pf|x_0-z_l|\\
  & &  +\ \E\co \left. (\tau_{i+1}-\tau_i)\mathbb 1_{ \tau_{i+1} \neq  \tau_C} \right|  \tau_i, \po X_{\tau_i},Y_{\tau_i}\pf = (x_0,-1)  \cf.
 \end{eqnarray*}
 Since the cooling schedule is non-increasing, the probability to reach the boundary and the expectation of the time of jump are bounded by the corresponding quantities at constant temperature $\varepsilon_{\tau_i}$. 
 Indeed, if $E \sim \mathcal E(1)$, conditionally on $ \tau_i$ and to $\po X_{\tau_i},Y_{\tau_i}\pf = (x_0,-1)$,
 \begin{eqnarray*}
\mathbb P\po \tau_{i+1} = \tau_C\pf & = & \mathbb P \po E \geq \int_{0}^{x_0-z_l} \frac{U'(x_0 - s)}{\varepsilon_{{\tau_i}+s}} ds\pf \\
  & \leq & \mathbb P \po E \geq  \frac{1}{\varepsilon_{\tau_i}}\int_{0}^{x_0-z_l} U'(x_0 - s) ds\pf \\
  & = & e^{-\frac{d}{\varepsilon_{\tau_i}}}
 \end{eqnarray*}
and
\begin{eqnarray*}
 \E\co (\tau_{i+1}-\tau_i)\mathbb 1_{ \tau_{i+1} \neq  \tau_C}  \cf & = & 2\int_{0}^{x_0-z_l} \mathbb P\po \frac{\tau_{i+1}-\tau_i}2 \geq u\pf du\\ 
  & = &  2\int_{0}^{x_0-z_l} \mathbb P\po E \geq \int_{0}^{u} \frac{U'(x_0 - s)}{\varepsilon_{{\tau_i}+s}} ds \pf du\\ 
 & \leq &  2\int_{0}^{x_0-z_l} \mathbb P\po E \geq  \frac{1}{\varepsilon_{\tau_i}}\int_{0}^{u} U'(x_0 - s) ds \pf du\\ 
  & = & 2 \int_{0}^{x_0-z_l} \frac{s U'(x_0 - s)}{\varepsilon_{\tau_i}} e^{-\frac{U(x_0-s)-U(x_0)}{\varepsilon_{\tau_i}}}ds.
 \end{eqnarray*}
 Thanks to Lemma \ref{LemRacineEpsi} there exist $\varepsilon_0$ and $c_1$ such that for all $\varepsilon \leq \varepsilon_0$, 
 \[e^{-\frac{d}{\varepsilon}}|x_0-z_l| + 2\int_{0}^{x_0-z_l} \frac{s U'(x_0 - s)}{\varepsilon} e^{-\frac{U(x_0-s)-U(x_0)}{\varepsilon}}ds \leq c_1 \sqrt{\varepsilon},\]
 which concludes the proof since $\varepsilon_{\tau_{i}} \leq \varepsilon_{t_0} \leq \varepsilon_0 $.
\end{proof}

\begin{proof}[Proof of Part \eqref{PartieUnTheoremeSortie} of Theorem \ref{TheoremsortieCoupe} for a simple cusp]
By choosing $\varepsilon_0$ small enough, for all cooling schedule $\varepsilon_t \leq \varepsilon_0$, the map
\[s \mapsto \po \varepsilon_s\pf^{-\frac12} e^{-\frac{d}{\varepsilon_s}}\]
is non-increasing, so that
\begin{eqnarray*}
 \E\co \left. \int_{\tau_i}^{\tau_{i+1}} \po \varepsilon_s\pf^{-\frac12} e^{-\frac{d}{\varepsilon_s}}ds\right| \tau_i,\ \bar A_C(\tau_i)\cf & \leq & \po \varepsilon_{\tau_i}\pf^{-\frac12} e^{-\frac{d}{\varepsilon_{\tau_i}}}\E\co \left. \tau_{i+1}-\tau_i \right|  \tau_i,\ \bar A_C(\tau_i)\cf \\
 & \leq & c_1 e^{-\frac{d}{\varepsilon_{\tau_i}}}.
\end{eqnarray*}
We define
\[\Phi(i,j) = \E\co \left.\int_{\tau_i}^{\tau_j} \po \varepsilon_s\pf^{-\frac12} e^{-\frac{d}{\varepsilon_s}}\mathbb 1_{s<\tau_C}ds\right| \tau_{i-1},\ \bar A_C(\tau_{i-1})\cf.\]
Obviously $\Phi(j,j)=0$; suppose  $\Phi(i+1,j) \leq c_1$ has already been proved for some $i< j$. Under the event $\bar A_C(\tau_{i-1})$
, there are two possibilities: either the process escapes the cusp $C$ between the times $\tau_{i-1}$ and $\tau_i$, in which case $\tau_i = \tau_C = \tau_{j}$ and the integral appearing in the definition of $\Phi(i,j)$ vanishes; or the attempt to exit $C$ between  $\tau_{i-1}$ and $\tau_i$ fails and the process returns  to $X_{\tau_i}=x_0$. To sum up,
\begin{eqnarray*}
 \Phi(i,j) &  =& \mathbb P\po \tau_i < \tau_C\ |\ \tau_{i-1},\ \bar A_C(\tau_{i-1})\pf \E\co \left.\int_{\tau_i}^{\tau_j} \po \varepsilon_s\pf^{-\frac12} e^{-\frac{d}{\varepsilon_s}}\mathbb 1_{s<\tau_C}ds\right| \tau_{i-1},\ \bar A_C(\tau_{i})\cf\\
 & \leq &\po1-e^{-\frac{d}{\varepsilon_{\tau_{i-1}}}}\pf \E\co \E\co \left.\left.\int_{\tau_i}^{\tau_{i+1}} \po \varepsilon_s\pf^{-\frac12} e^{-\frac{d}{\varepsilon_s}}\mathbb 1_{s<\tau_C}ds\right| \tau_{i},\ \bar A_C(\tau_{i})\cf + \Phi(i+1,j)\right| \ \tau_{i-1}\cf\\
 & \leq &  \po1-e^{-\frac{d}{\varepsilon_{\tau_{i-1}}}}\pf\E \co\left.  c_1 e^{-\frac{d}{\varepsilon_{\tau_i}}}+ c_1\right| \ \tau_{i-1} \cf\\
 & \leq & \po1-e^{-\frac{d}{\varepsilon_{\tau_{i-1}}}}\pf\po  c_1 e^{-\frac{d}{\varepsilon_{\tau_{i-1}}}}+c_1\pf \leq c_1.
\end{eqnarray*}
 This proves by induction that $\Phi(1,j) \leq c_1$ for all $j\geq 1$. As was already pointed out there can only be a finite number of jumps in a finite time, so that the sequence $(\tau_i)_{i\geq 0}$ almost surely converges to $\tau_C$ (which, at this point, may be infinite). The monotone convergence theorem yields
 \[\E\co \left. \int_{\tau_1}^{\tau_C} \po \varepsilon_s\pf^{-\frac12} e^{-\frac{d}{\varepsilon_s}}ds \right| \bar A_C(t_0)\cf  \leq c_1.\]
 On the other hand, as soon as $\varepsilon_t\leq \varepsilon_0$,
\[\E\co \left. \int_{t_0}^{\tau_1} \po \varepsilon_s\pf^{-\frac12} e^{-\frac{d}{\varepsilon_s}}ds \right| \bar A_C(t_0)\cf  \leq \po \varepsilon_0\pf^{-\frac12} e^{-\frac{d}{\varepsilon_0}} \max\po |x_0-z_l|,\ |z_r-x_0|\pf.\]
Bringing all the pieces together, there exists a constant $c>0$ such that 
\[\E\co \left. \int_{t_0}^{\tau_C} \po \varepsilon_s\pf^{-\frac12} e^{-\frac{d}{\varepsilon_s}}ds \right| \bar A_C(t_0)\cf  \leq c.\]
\end{proof}
Note that if the cooling schedule is such that $\int_{t_0}^{\infty} \po \varepsilon_s\pf^{-\frac12} e^{-\frac{d}{\varepsilon_s}}ds = \infty$, this first result implies that $\tau_C$ is almost surely finite, in which case $(X_{\tau_C},Y_{\tau_C})$ is well defined, and the second half of Theorem \ref{TheoremsortieCoupe} makes sense.

\begin{proof}[Proof of Part \eqref{PartieDeuxTheoremeSortie} of Theorem \ref{TheoremsortieCoupe} for a simple cusp]
 We note $p_i^l$ (resp. $p_i^r$) the probability to exit $C$ in one shot, meaning before reaching again $x_0$, starting at $(X_{\tau_i},Y_{\tau_i}) = (x_0,-1)$ (resp. $(x_0,1)$), with $X_{\tau_C} = z_l$ (resp. $z_r$). Namely, if $E\sim\mathcal E(1)$,
 \begin{eqnarray*}
 p_i^l & =& \mathbb P\po \int_0^{x_0-z_l} \frac{1}{\varepsilon_{\tau_i+s}}U'(x_0- s)ds \leq E\pf \\
 p_i^r & =& \mathbb P\po \int_0^{z_r-x_0} \frac{1}{\varepsilon_{\tau_i+s}}U'(x_0 + s)ds \leq E\pf.
 \end{eqnarray*}
 Since the cooling schedule is not-increasing,
\[e^{-\frac{d}{\varepsilon_{\tau_{i+1}}}} \leq p_i^l, p_i^r \leq e^{-\frac{d}{\varepsilon_{\tau_i}}}.\]
In particular $p_i^l \geq p_{i+1}^r$, which means that between two consecutive attempts to exit the cusp, the second one is always less likely to succeed than the first one, 
and thus if $(X_{\tau_i},Y_{\tau_i}) = (x_0,-1)$ then the probability that the exit point will be $z_l$ is greater than $\frac12$. On the other hand,
\begin{eqnarray*}
\mathbb P\po X_{\tau_C} = z_l \ |\ (X_{\tau_1},Y_{\tau_1}) = (x_0,1)\pf  & =& (1-p_1^r)\mathbb P\po X_{\tau_C} = z_l \ |\ (X_{\tau_2},Y_{\tau_2}) = (x_0,-1)\pf\\
& \geq & \frac12 \po 1 - e^{-\frac{d}{\varepsilon_{t_0}}}\pf.
\end{eqnarray*}
Under the event $\bar A_C(t_0)$, necessarily $\tau_1 < \tau_C$, in other words $X_{\tau_1} = x_0$ and $Y_{\tau_1}$ is either 1 or -1. Since the previous arguments cover both cases, as soon as $\varepsilon_{t_0} \leq \varepsilon_0$,
\begin{eqnarray*}
\mathbb P\po X_{\tau_C} = z_l \ |\ \bar A_C(t_0)\pf  & \geq & \frac12 \po 1 - e^{-\frac{d}{\varepsilon_{0}}}\pf.
\end{eqnarray*}
The case of $z_r$ is symmetric.
\end{proof}

\subsubsection{Induction}

In this section we consider a general cusp of depth $d$ and we prove Theorem \ref{TheoremsortieCoupe} under the induction assumption that it holds for any cusp with strictly fewer local maxima than $C$. There is a finite number of maxima $z$ in $C$ for which $u=U(z_l)-U(z)$, and each connected component of $(x_l,x_r)\smallsetminus\{z,\ u=U(z_l)-U(z)\}$ is a cusp of depth at most $g = d-u$ (cf Fig. \ref{FigureInduction}). We call $(C_i)_{i=1..n}$ these cusps of depth $g $. We can consider $\varepsilon_0$ and $M$ such that for all $\varepsilon_t\leq \varepsilon_0$, for all $i=1,\dots,n$, for all $z_i$ which is an end of $C_i$ and for all $t_0$,
\begin{eqnarray*}
\mathbb P\po \left. X_{\tau_{C_i}} = z_i\right| \  \bar A_{C_i}(t_0)\pf  & \geq&  \frac1M\\
 \text{ and }\hspace{15pt}\E\co\left. \int_{t_0}^{\tau_{C_i}} \po \varepsilon_s\pf^{-\frac12} e^{-\frac{g}{\varepsilon_s}}ds\right| \bar A_{C_i}(t_0) \cf  & \leq & M\\
 \Rightarrow\hspace{15pt}\E\co\left. \int_{t_0}^{\tau_{C_i}} \po \varepsilon_s\pf^{-\frac12} e^{-\frac{d}{\varepsilon_s}}ds\right| \bar A_{C_i}(t_0) \cf  & \leq & M e^{-\frac{u}{\varepsilon_{t_0}}}.
\end{eqnarray*}

\begin{figure}
 \centering
 \includegraphics[scale=0.5]{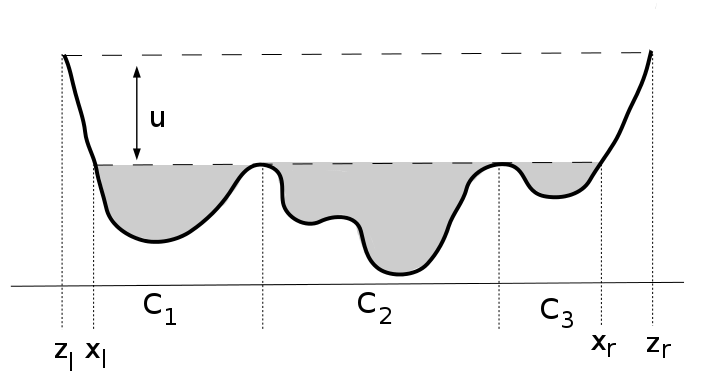}
\caption{The cusp $C$ is divided in several smaller cusps.}\label{FigureInduction}
\end{figure}

\begin{proof}[Proof of Theorem \ref{TheoremsortieCoupe}, part \eqref{PartieUnTheoremeSortie}]
Set $J_0 = t_0$ and define $K_i$, $J_i$ for $i\geq 1$ by
\begin{eqnarray*}
 K_i &  =& \min\po \inf\{t>J_{i-1},\ A_C(t) \},\tau_C\pf\\
 J_i & =& \min\po \inf\{t>K_{i},\ \bar A_C(t) \},\tau_C\pf.
\end{eqnarray*}
At time $K_i$, the process is ready to escape: it is located near the boundary of $C$, and its velocity is directed toward the exit. At time $J_i$, if it failed to leave the cusp, it goes back to the $C_j$'s (see Fig. \ref{FigureInduction}) and, as previously,
\begin{eqnarray*}
 \Phi(i,j) & := & \E \co\left. \int_{K_i}^{K_j}  \po \varepsilon_s\pf^{-\frac12} e^{-\frac{d}{\varepsilon_s}}ds\right| K_{i-1},\ \tau_C > K_{i-1} \cf\\ 
 &  =& \mathbb P\po\left.\tau_C > K_{i}\right| K_{i-1},\ \tau_C > K_{i-1}\pf \E \co\left. \int_{K_i}^{K_j}  \po \varepsilon_s\pf^{-\frac12} e^{-\frac{d}{\varepsilon_s}}ds\right| K_{i-1},\ \tau_C > K_{i} \cf\\
 & \leq &\po1-e^{-\frac{u}{\varepsilon_{K_{i-1}}}}\pf\E \co\left. \int_{K_i}^{J_i}+\int_{J_i}^{K_{i+1}} +\int_{K_{i+1}}^{K_j}   \po \varepsilon_s\pf^{-\frac12} e^{-\frac{d}{\varepsilon_s}}ds\right| K_{i-1},\ \tau_C > K_{i} \cf.
\end{eqnarray*}
We treat the three parts of the integral one after another. From Lemma \ref{LemControleRacine}, there exists a constant $c_1$ (depending on $C$ but not on $(\varepsilon_t)_{t\geq t_0}$) such that
\[\E \co\left. \int_{K_i}^{J_i}   \po \varepsilon_s\pf^{-\frac12} e^{-\frac{d}{\varepsilon_s}}ds\right| K_i,\ \tau_C > K_{i} \cf \leq c_1 e^{-\frac{u}{\varepsilon_{K_i}}}.\]

As far as the time interval $[J_i,K_{i+1}]$ is concerned, we distinguish the following cases: if $J_i = \tau_C$, the process has escaped from $C$, hence $K_{i+1}=\tau_C$ and the integral vanishes. Otherwise the process goes back down to a cusp $C_i$, and then for a while pass from a $C_j$ to another, until it reaches $A_C$ again. The contribution of the time passed in each $C_i$ to the integral can be controlled thanks to the induction assumption. It remains to make sure the number of transitions between the $C_j$'s before existing is not too large. When the process exits $C_i$, the probability it does so through a given end is bounded below thanks to the induction assumption. Thus the expectation of the number $N$ of jumps from a $C_i$ to another $C_j$ before the process reaches $A_C$ is bounded by a constant $D$ which does not depend on the cooling schedule. Hence, denoting by  $\tau_{C_i}^*$ an entry time of the process in $C_i$,
\begin{eqnarray*}
 & & \E \co\left. \int_{J_i}^{K_{i+1}}   \po \varepsilon_s\pf^{-\frac12} e^{-\frac{d}{\varepsilon_s}}ds\right| K_i,\ \tau_C > K_{i}\cf \\
 & \leq & \E \co\left. N \underset{i}\max \E \co\left. \int_{\tau_{C_i}^*}^{\tau_{C_i}}   \po \varepsilon_s\pf^{-\frac12} e^{-\frac{d}{\varepsilon_s}}ds\right| \tau_{C_i}^*>K_i,\ \bar A_{C_i}\po \tau_{C_i}^*\pf \cf \right| K_i,\ \tau_C > J_{i}\cf\\
 & \leq & DM e^{-\frac{u}{\varepsilon_{K_i}}}.
\end{eqnarray*}
Since $\Phi(j,j) = 0$, as an induction assumption we can suppose
\[\Phi(i+1,j) \leq DM+c_1,\]
and thus we get
\begin{eqnarray*}
 \Phi(i,j) & \leq & \po1-e^{-\frac{u}{\varepsilon_{K_{i-1}}}}\pf \E \co \left. e^{-\frac{u}{\varepsilon_{K_i}}}(DM+c_1) + \Phi(i+1,j) \right| K_{i-1},\ \tau_C > K_{i}\cf.\\
 & \leq & \po1-e^{-\frac{u}{\varepsilon_{K_{i-1}}}}\pf \po e^{-\frac{u}{\varepsilon_{K_{i-1}}}} +1\pf \po DM+c_1\pf\\
 & \leq & DM+c_1.
\end{eqnarray*}
The monotone convergence Theorem yields
\[ \E \co\left. \int_{K_1}^{\tau_C} \po \varepsilon_s\pf^{-\frac12} e^{-\frac{d}{\varepsilon_s}}ds\right| \ \bar A_C(t_0) \cf \leq DM+c_1.\]
On the other hand in the event $\bar A_C(t_0)$, either the process start in a $C_i$, or it reaches a $C_i$ in a time bounded by $\max(|x_l-z_l|,\ |z_r-x_r|)$. In both cases with the previous argument we used for $[J_i,K_{i+1}]$,
\begin{eqnarray*}
  \E \co\left. \int_{t_0}^{K_{1}}   \po \varepsilon_s\pf^{-\frac12} e^{-\frac{d}{\varepsilon_s}}ds\right| \bar A_C(t_0)\cf & \leq & DM +  \po \varepsilon_{t_0}\pf^{-\frac12} e^{-\frac{d}{\varepsilon_{t_0}}} \max(|x_l-z_l|,\ |z_r-x_r|),
\end{eqnarray*}
and ultimately part \eqref{PartieUnTheoremeSortie} of Theorem \ref{TheoremsortieCoupe} is proved with
\[c = 2DM+c_1 + \po \varepsilon_{t_0}\pf^{-\frac12} e^{-\frac{d}{\varepsilon_{t_0}}} \max(|x_l-z_l|,\ |z_r-x_r|). \]
\end{proof}

The same remark as in Section \ref{SubsectionCuspSimple} holds: when $\int_{t_0}^{\infty} \po \varepsilon_s\pf^{-\frac12} e^{-\frac{d}{\varepsilon_s}}ds = \infty$, the first part of Theorem \ref{TheoremsortieCoupe} implies that $\tau_C$ is almost surely finite, and $X_{\tau_C}$ well-defined.

\begin{proof}[Proof of Theorem \ref{TheoremsortieCoupe}, part \eqref{PartieDeuxTheoremeSortie}]
 The situation is very similar to the simple cusp one. Take $M'_0 = t_0$ and
 \begin{eqnarray*}
 M_{i} &= & \min \po \inf\{t>M'_i,\ A_C(t)\},\ \tau_C\pf\\
 M'_{i+1} &= & \min \po \inf\{t>M_i,\ \bar A_C(t)\},\ \tau_C\pf.
\end{eqnarray*}
Given the sequence $(M_i)_{i\geq 1}$, $(X_{M_i})_{i\geq 1}$ is an inhomogeneous Markov chain in $\{z_l,x_l,x_r,z_r\}$. 
Let
\begin{eqnarray*}
 p_i^l & =& \mathbb P\po \int_0^{x_l-z_l} \frac{1}{\varepsilon_{M_i+s}}U'(x_0- s)ds \leq E\pf ,\\
 p_i^r & =& \mathbb P\po \int_0^{z_r-x_r} \frac{1}{\varepsilon_{M_i+s}}U'(x_0 + s)ds \leq E\pf.
 \end{eqnarray*}
The induction assumption on the $C_j$'s implies that the transitions of this chain from $x_r$ to $x_r$ and vice-versa are bounded below by a constant $h>0$ which does not depend on the cooling schedule. Hence $(X_{M_i})_{i\geq 1}$ is more likely to reach $z_l$ before $z_r$ than the chain $(\tilde X_{M_i})_{i\geq 1}$ with transition
\begin{displaymath}
 \begin{array}{lclclcl}
   \mathbb P(x_l \rightarrow x_r ) & = & 1 - p_i^l & \hspace{25pt} & \mathbb P(x_r \rightarrow x_l ) & = & h \\
   \mathbb P(x_l \rightarrow z_l ) & = & p_i^l & \hspace{25pt} & \mathbb P(x_r \rightarrow z_r ) & = & p_i^r\\
   \mathbb P(x_l \rightarrow x_l ) & = & 0 & \hspace{25pt} & \mathbb P(x_r \rightarrow x_r ) & = & 1 - h - p_i^r.
 \end{array}
\end{displaymath}
Similarly to the situation in Section \ref{SubsectionCuspSimple}, $p_i^l \geq p_j^r$ if $j>i$, and $p_1^r \leq e^{-\frac{u}{\varepsilon_0}}$. The probability that $\tilde X$, starting from $\tilde X_{M_{i+1}}=x_r$, hits $z_r$ before $x_l$ is 
\[\sum_{j\geq 1} p_{i+j}^r \prod_{1\leq k < j} (1-h-p_k^r) \leq \sum_{j\geq 1} p_{i+1}^r  (1-h)^{j-1} \leq \frac{p_i^l}{1-h}. \]
In particular this is less than $\frac{e^{-\frac{u}{\varepsilon_0}}}{1-h}$, so that the probability that there exist $k_0\geq 0$ such that $\tilde X_{M_{k_0}} = x_l$ is bounded below. Hence to bound the probability to reach $z_l$ before $z_r$ we can assume the initial point is $x_l$. Let
\[k_{i+1} = \inf\{n > k_i,\ \tilde X_{M_n} \neq x_r\}.\]
Given the times $\po M_{k_i}\pf$, the chain $\tilde X$ is more likely to reach $z_l$ before $z_r$ than the chain $X'$ with transition 
\[\mathbb P(x_l \rightarrow z_l )  = p_{k_i}^l  \hspace{25pt} \mathbb P(x_l \rightarrow z_r )  =  \frac{p_{k_i}^l}{1-h}  \hspace{25pt} \mathbb P(x_l \rightarrow x_l )  = 1 - p_{k_i}^l - \frac{p_{k_i}^l}{1-h}.\]
Finally, $X'$ goes to $z_l$ rather than $z_r$ with a probability $\frac{1}{1+\frac{1}{1-h}}$, which does not depend on the cooling schedule, which concludes.
\end{proof}

\subsection{Theorem \ref{TheoremLowerBound}}

We start by some preliminary lemmas. Let $C$ be a cusp of depth $d$, recall that its bottom is $B = \{x\in C, U(z)\geq U(x) \ \forall z\in C\}$ and let $x_b = \min B$. We will give an upper bound of the probability that the process, starting at time $t_0$ at point $(x_b,-1)$, reaches $z_l$ before $(x_b,1)$. Obviously, if there were no local maximum between $x_b$ and $z_l$, this probability would be bounded by $e^{-\frac{d}{\varepsilon_{t_0}}}$. To prove a similar bound in more general cases, we decompose $C$ in the following way: let
\[J = \{ x\in [z_l,x_b],\text{ s.t. } \forall z\in (x,x_b],\ U(z)<U(x) \}.\]
On the set $J$, $U$ is non-increasing, and the connected component of the closure of $[z_l,x_b] \smallsetminus J$ are  cusps whose right end are local maxima. We call $u_1 > u_2 > \dots > u_{q}$ this local maxima, $C_1,\dots,C_{q-1}$ the corresponding cusps, $u_0 = x_b$, $C_0 = \{x_b\}$, $u_{q+1}=z_l$ and $C_{q+1} = \{ z_l\}$. 
The point of taking the connected component of \emph{the closure} of $[z_l,x_b] \smallsetminus J$ was to ensure $U(u_i) > U(u_j)$ if $i>j$. We say $U(u_i)$ is the energy level of $C_i$, denoted by $E_i$, and we note
\[\delta_i = E_{i+1}-E_i.\]
Note that $\sum_{i=0}^q \delta_i = E_{q+1} - E_0 = d$. The situation is represented in Figure \ref{FigureTheoremSix}.
Let $t_0>0$, $1\leq i < q$ and suppose $\bar A_{C_i}(t_0)$ holds. Let
\[s = \inf \{ t > t_0 \text{ s.t. }\exists j \neq i,\ X_t \in C_j\}.\]
\begin{lem}
 There exists $c_2 >0$ which depends on $C$ but not on the cooling schedule so that if $\bar A_{C_i}(t_0)$ holds,
 \begin{eqnarray}\label{EqEi}
   \mathbb P\po U(X_{s})  = E_{i+1}\pf & \leq & c_2 e^{-\frac{\delta_i}{\varepsilon_{t_0}}}.
 \end{eqnarray}
\end{lem}
\begin{proof}
 Since there is a finite number of cusps $C_i$, it is enough to show such a constant exists for any one of them. Starting from $A_{C_i}$, if the process exits $C_i$ to the right, it will deterministically falls down to $C_{i-1}$, which means $U(X_s) = E_{i-1}$. If the process exits to the left, it reaches $C_{i+1}$ in one shot with probability less than $e^{-\frac{\delta_i}{\varepsilon_{t_0}}}$, else goes back down to $\bar A_{C_i}$. Since the probability the process escapes to the left, starting from $\bar A_{C_i}$, is bounded by a constant $h>0$ thanks to Theorem \ref{TheoremsortieCoupe},
 \[\mathbb P\po U(X_{s})= E_{i+1}\pf \leq \sum_{k\geq 0}(1-h)^k e^{-\frac{\delta_i}{\varepsilon_{t_0}}} = \frac{e^{-\frac{\delta_i}{\varepsilon_{t_0}}}}{1-h}.\]
\end{proof}
\textbf{Remark:} \emph{The constant $c_2$ is defined from $h$, which depends only on the $C_i$'s. Thus If $(D_l)_{l\geq0}$ is a family of nested cusps such that for all $k\geq 0$ all the minima of $D_k$ belongs to $\underset{l\geq0}\bigcap D_l$, the constant $c_2$ may be defined uniformly on $k$ so that  \eqref{EqEi} holds for all $D_k$'s. Since $q$ (the number of small cusps) is also the same for all $D_k$'s, this remark will extend to the next result}.

This lemma implies the process, starting from $\bar A_{C_i}$, is less likely to hit $(z_l,-1)$ before $(x_0,1)$ than the birth-death process $(W_n)_{n\geq0}$ on $[0,q+1]$ with transition probabilities
\[\mathbb P(j\rightarrow j+1) = c_2 e^{-\frac{\delta_j}{\varepsilon_{t_0}}}, \hspace{30pt}\mathbb P(j\rightarrow j-1) = \po1- c_2 e^{-\frac{\delta_j}{\varepsilon_{t_0}}}\pf,\hspace{20pt}j=\in\cco 1,q\ccf,\]
is to hit $q+1$ before $0$, starting at $i$.

\begin{figure}
 \centering
 \includegraphics[scale=0.4]{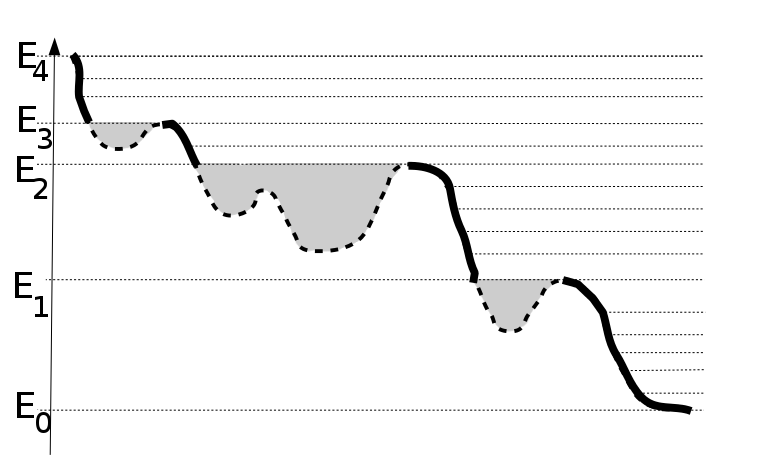}
\caption{If the light comes from the right, $J$, in bold line, is the sunny part, and the $C_i$'s are in the shadow.}\label{FigureTheoremSix}
\end{figure}

\begin{lem}\label{LemSortie}
For $\varepsilon_{t_0}$ small enough,
 \[\mathbb P\po W \text{ hits }q+1\text{ before } 0\ |\ W_0 = 1\pf \leq \po 2c_2\pf^q e^{-\frac{E_{q+1}-E_1}{\varepsilon_{t_0}}}.\]
 As a consequence,
 \[\mathbb P\po X \text{ hits }z_l\text{ before } x_b\ |\ \bar A_{C_1}(t_0)\pf \leq \po 2c_2\pf^q e^{-\frac{E_{q+1}-E_1}{\varepsilon_{t_0}}}\]
 and
 \[\mathbb P\po (X,Y) \text{ hits }(z_l,-1)\text{ before } (x_b,1)\ |\ (X_{t_0},Y_{t_0}) = (x_b,-1)\pf \leq \po 2c_2\pf^q e^{-\frac{d}{\varepsilon_{t_0}}}.\]
\end{lem}
\begin{proof}
 More generally, let
 \[r_i = \mathbb P\po W \text{ hits }q+1\text{ before } i-1\ |\ W_0 = i\pf. \]
 If $W$ hits $q$ before $i-1$, necessarily $W_1 = i+1$ (which occurs with probability $c_2e^{-\frac{\delta_i}{\varepsilon_{t_0}}}$) and either the chain stays above $i$ (with probability $r_{i+1}$), either it goes back at some point to $i$ and then we are back to the initial problem. Thus
 \[r_i =c_2e^{-\frac{\delta_i}{\varepsilon_{t_0}}}\po r_{i+1} + (1-r_{i+1})r_i\pf. \]
 Suppose $\varepsilon_{t_0}$ is small enough to have $c_2e^{-\frac{\delta_i}{\varepsilon_{t_0}}} \leq \frac12$ for all $i=1,..,q$, so that
 \[r_i \leq 2 c_2e^{-\frac{\delta_i}{\varepsilon_{t_0}}} r_{i+1},\]
 which concludes since $r_{q+1}=1$.
\end{proof}
We only considered here the left part of $C$, but the same goes for the right one. If $x_b' = \max B$ then there exist $c_2'$ and $q'$ such that
\[\mathbb P\po (X,Y) \text{ hits }(z_r,1)\text{ before } (x_b',-1)\ |\ (X_{t_0},Y_{t_0}) = (x_b',1)\pf \leq \po 2c_2'\pf^{q'} e^{-\frac{d}{\varepsilon_{t_0}}},\]
and we write $c_3 = \max\po \po 2c_2\pf^q,\ \po 2c_2'\pf^{q'}\pf$.

\begin{proof}[Proof of Theorem \ref{TheoremLowerBound}]
Let $a>0$ be small enough such that $U$ is non-decreasing (resp. non-increasing) on $\{ x_b-a,x_b\}$ (resp. $\{ x_b',x_b'+a\}$) and that $U(z_l) - \max\po U(x_b-a),U(x_b'+a)\pf \geq \frac d2$. That way,
\[\mathbb P\po (X,Y) \text{ hits }(z_l,1)\text{ before } (x_b,1)\ |\ (X_{t_0},Y_{t_0}) = (x_b-a,-1)\pf \leq c_3 e^{-\frac{d}{2\varepsilon_{t_0}}}\]
(and similarly in the right part of $C$). It means if at some time $t_1$ the temperature has been divided by two since $t_0$, then from $t_1$ we have the same bound on the probability of success of an attempt to leave $C$ starting from $(x_b-a,-1)$ than we had for an attempt to leave $C$ starting from $(x_b,-1)$ at initial temperature.

Let $t_j = t_0 + 2aj$ and $K_0 = t_0$, suppose $K_j$ has been defined for some $j\in\mathbb N$, let $S_{0,j} = K_j$ and for $i\geq 0$
\begin{eqnarray*}
 \widetilde S_{i+1,j} &=&  \inf\left\{t > S_{i,j} \text{ s.t. } \po (X_t,Y_t) \in \{(x_b,-1),(x_b',1)\}\text{ and } \varepsilon_t \geq \frac12 \varepsilon_{t_j}\pf \right\}\\
 \overline S_{i+1,j} &=&  \inf\left\{t > S_{i,j} \text{ s.t. } \po (X_t,Y_t) \in \{(x_b-a,-1),(x_b'+a,1)\}\text{ and } \varepsilon_t < \frac12 \varepsilon_{t_j}\pf \right\}\\
 S_{i+1,j} &=& \min\po \tau_C,\widetilde S_{i+1,j}, \overline S_{i+1,j}\pf.
\end{eqnarray*}
Let $N_j = 0$ if $K_j = \tau_C$ and else let
\begin{eqnarray*}
N_j & = & \inf\left\{n\geq 1, \ S_{n,j} \geq K_j  + 2a\text{ or }S_{n,j}=\tau_C\right\}\\
K_{j+1} & = & S_{N_j,j}.
\end{eqnarray*}
The situation is the following: while the process has not escaped $C$ yet, the $S_{i,j}$'s are starting times of new attempts to leave with, for all $i,j\in\mathbb N^2$
\begin{eqnarray*}
\mathbb P\po S_{i+1,j}=\tau_C  \ | \ S_{i,j},\ S_{i,j}<\tau_C\pf & \leq &  c_3 e^{-\frac{d}{2\varepsilon_{S_{i,j}}}}.
\end{eqnarray*}
A problem is that we don't control the way the temperature evolves with $i,j$, and this is why we introduced the $K_j$'s, which are the times at which we update the temperature. Indeed $K_j \geq t_j$ so that for all $i,j\in\mathbb N^2$, $\varepsilon_{S_{i,j}} \leq \varepsilon_{t_j}$, which implies $\po \mathbb 1_{S_{i\wedge N_j,j}<\tau_C} + (i\wedge N_j)  c_3 e^{-\frac{d}{2\varepsilon_{t_j}}}\pf_{i\geq 0}$ is a submartingale and
\begin{eqnarray*}
\mathbb P\po K_{j+1}<\tau_C  \ | \ K_j,\ K_{j}<\tau_C\pf & = & 1+\mathbb E\po \mathbb 1_{K_{j+1}<\tau_C} - 1_{K_{j}<\tau_C}  \ | \ K_j,\ K_{j}<\tau_C\pf \\
&\geq & 1-  c_3 e^{-\frac{d}{2\varepsilon_{t_j}}}\mathbb E\po N_j  \ | \ K_j,\ K_{j}<\tau_C\pf.
\end{eqnarray*}
In the first instance suppose there exists $m$ (which does not depend on the cooling schedule) such that for all $j\geq 0$
\begin{eqnarray}\label{EqENj}
\mathbb E\po N_j  \ | \ K_j,\ K_{j}<\tau_C\pf & \leq & \frac{m}{\sqrt{\varepsilon_{t_j}}}.
\end{eqnarray}
In that case by induction we get
\begin{eqnarray*}
\mathbb P\po K_{j+1}<\tau_C \pf & =& \mathbb E\po \mathbb P\po K_{j+1}<\tau_C  \ | \ K_j,\ K_{j}<\tau_C\pf \mathbb P\po K_{j }<\tau_C \pf\pf\\ 
&\geq & \po 1-  \frac{m c_3}{\sqrt{\varepsilon_{t_j}}} e^{-\frac{d}{2\varepsilon_{t_0+j}}}\pf \mathbb P\po K_{j}<\tau_C \pf\\
& \geq & \prod_{k=0}^{j} \po 1-  \frac{m c_3}{\sqrt{\varepsilon_{t_k}}} e^{-\frac{d}{2\varepsilon_{t_k}}}\pf .
\end{eqnarray*}
Suppose $\varepsilon_{t_0}$ is small enough so that $ m c_3\po \varepsilon_{t_0}\pf^{-\frac12} e^{-\frac{d}{2\varepsilon_{t_0}}}$ is less than the positive solution of $1-z = e^{-2z}$, and so that $\varepsilon \mapsto \varepsilon^{-\frac12} e^{-\frac{d}{\varepsilon}}$ is increasing on $(0,\varepsilon_{t_0})$. Hence for $j\geq1$
\begin{eqnarray*}
\mathbb P\po t_j <\tau_C \pf & \geq & \mathbb P\po K_{j+1}<\tau_C \pf \\
& \geq & \exp\po - 2m c_3 \sum_{k=0}^{j-1} \frac{e^{-\frac{d}{2\varepsilon_{t_k}}}}{\sqrt{\varepsilon_{t_k}}} \pf\\
& \geq & \exp\po {-\frac{mc_3}a \int_{t_0}^{t_j} \frac{e^{-\frac{d}{\varepsilon_{u}}}}{\sqrt{\varepsilon_{u}}} du }\pf,
\end{eqnarray*}
and for all $r>0$,
\begin{eqnarray*}
\mathbb P\po r <\tau_C \pf & \geq & \mathbb P\po t_0 + 2a\left\lceil \frac{r- t_0}{2a}\right\rceil <\tau_C \pf  \\
& \geq & \exp\po {- 2mc_3\po \frac{e^{-\frac{d}{\varepsilon_{t_0}}}}{\sqrt{\varepsilon_{t_0}}}+ \frac{1}{2a}\int_{t_0}^{r} \frac{e^{-\frac{d}{\varepsilon_{u}}}}{\sqrt{\varepsilon_{u}}} du \pf}\pf.
\end{eqnarray*}
It remains to prove \eqref{EqENj}, which states that between two consecutive updates of the temperature, there are not too many attempts to leave. Intuitively, if the temperature decays slowly then the inter-jump times are of order $\sqrt{\varepsilon_{t_j}}$ and so the number of attempts in a fixed duration should of the order  $ {\varepsilon_{t_j}}^{-\frac12}$. On the other hand if the temperature falls rapidly then we only take into account attempts that starts from $(x_b-a,-1)$ or $(x_b+a,1)$, and between two such attempts the process has to cover a distance of at least $2a$.

More precisely, since the process moves at constant speed, for all $i,j$, $\overline S_{i+1,j} - \overline S_{i,j}> 2a$, so that 
\[N_j \leq 1 + \inf\left\{n\geq 1, \ \widetilde S_{n,j} \geq K_j  + 2a\text{ or }\varepsilon_{\widetilde S_{n,j}} \leq \frac12\varepsilon_{t_j} \right\}. \]
Let $\rho>0$ be such that $U(x_b -s) \leq U(x_b) + \rho (x_b-s)^2$ for $s\leq |x_b-z_l|$ (and similarly in the right part of $C$; the situation being the same, we only consider the left part in the following). Starting from $(x_b,-1)$ at time $S_{i,j}$ the next jump time $T$ of the process, defined by \eqref{DefiTavecExpo} from an exponential r.v. $E$, satisfies,
\[T \geq \inf \left\{s>S_{i,j},\ E < \frac{\rho s^2}{\frac12\varepsilon_{t_j}} \right\} \wedge \inf \left\{s>S_{i,j},\ \varepsilon_{s} \leq \frac12\varepsilon_{t_j} \right\} \]
and thus $\widetilde S_{i+1,j} - S_{i,j} \geq 2 T$ and
\[N_j \leq 1 + \inf\left\{n\geq 1, \sqrt{\frac{\varepsilon_{t_j}}{2\rho}}\sum_{i=1}^n \sqrt{E_i} \geq 2a \right\} \]
where $(E_i)_{i\geq 1}$ is a sequence of i.i.d. random variables with law $\mathcal E(1)$, independent from the past ($t\leq K_j$). Now
\begin{eqnarray*}
N_j & \leq & 1+ \left\lceil\frac{4a \sqrt{\rho}}{\mathbb E\po \sqrt E\pf\sqrt{\varepsilon_{t_j}}}\right\rceil + \inf\left\{n\geq 1, \frac1n\sum_{i=1}^{n} \sqrt{E_i} \geq \frac12  \mathbb E\po\sqrt E\pf  \right\}.
\end{eqnarray*}
Note that $Z_k = \frac1k\sum_{i=1}^{k} \po \frac{\sqrt{E_i}}{ \mathbb E\po\sqrt E\pf} -1\pf$ satisfies a Large Deviation Principle, so that
\begin{eqnarray*}
\mathbb E\po N_j  \ | \ K_j,\ K_{j}<\tau_C\pf & \leq & 1+ \left\lceil\frac{4a \sqrt{\rho}}{\mathbb E\po \sqrt E\pf\sqrt{\varepsilon_{t_j}}}\right\rceil + \sum_{k\geq 1}  \mathbb P\po Z_k \leq -\frac12\pf  \\
&\leq & \frac{m}{\sqrt{\varepsilon_{t_j}}}
\end{eqnarray*}
for some $m<\infty$ as long as $\varepsilon_{t_0}$ is small enough.
\end{proof}
\textbf{Remark:} \emph{Here the constant $c_4$ only depends on $U''(x_b)$ and $U''(x_b')$, and $m$ only depends on $c_4$. Furthermore, as has already been noticed, $c_3$ only depends on the internal sub-cusps of $C$. Thus, if $(D_l)_{l\geq 0}$ is a family of cusps so that for all $k\geq 0$ all the minima of $D_k$ belong to $\underset{l\geq0}\bigcap D_l$, the constant $\Gamma$ in Theorem \ref{TheoremLowerBound}  may be chosen uniformly over $k$. And if $U$ is a potential with a finite number of local minima, going to $+\infty$ at $\pm \infty$, $\Gamma$ may be chosen uniformly over all cusps of $U$.}

\subsection{Proof of the NS condition}

Now that Theorems \ref{TheoremsortieCoupe} and \ref{TheoremLowerBound} are established, we recall (and slightly adapt) the arguments from \cite{Hajek} to prove Theorem \ref{TheoremNS}.

\bigskip

Let $E\geq 0$ and
\begin{eqnarray*}
 W_E & = & \left\{x\in\R,\ x\text{ is a local minimum of depth strictly larger than }E \right\}\\
 R_E & = & \left\{x\in\R,\ x\text{ is reachable from $y$ at height $V(y)+E$ for some }y\in W_E \right\}\\
 J & = & \left\{x\in\R,\ W_E\text{ is reachable from $x$ at height }V(x) \right\}.
\end{eqnarray*}
If $x\notin J$, then the set of points which are reachable from $x$ at height $V(x)$ is a cusp of depth at most $E$ (else its bottom would be in $W_E$ and reachable from $x$ at height $V(x)$) which does not intersect $J$ (if a point $y$ were in the intersection, we could reach $y$ from $x$ at height $V(x)$ and then reach a point in $W_E$ from $y$ at height $V(y)\leq V(x)$). 
Thus the connected component of  $\R\smallsetminus J$ which contains $x$ is itself a cusp of depth at most $E$. Moreover  one of its end is a local maximum (else it could be thicken without intersecting $J$) which means there are finitely many connected components of  $\R\smallsetminus J$; we call them $C_1,\dots,C_n$.

\begin{lem}\label{LemPegalUn}
Let $E>0$. 
If
 \[\int_0^\infty  \po \varepsilon_{s}\pf^{-\frac12}e^{-\frac{E}{\varepsilon_{s}}} ds = \infty \]
 then
 \[\underset{t\rightarrow\infty}\lim \mathbb P \po X_t \in R_E\pf = 1.\]
\end{lem}
\begin{proof}
Note that there are a finite number of minima in $W_E$, so that there exists $\gamma_0>0$ such that all those are of depth larger than $E+\gamma_0$. We thicken $R_E$, for $\gamma < \gamma_0$, as
\begin{eqnarray*}
 R_{E,\gamma} & = & \left\{x\in\R,\ x\text{ is reachable from $y$ at height $V(y)+E+\gamma$ for some }y\in W_E \right\}.
\end{eqnarray*}
Since $R_E = \underset{\gamma >0}\bigcap R_{E,\gamma}$, we only need to prove the result for an arbitrary $\gamma<\gamma_0$. Let $t_0 \geq 0$, $A_0 = t_0$ and
 \begin{eqnarray*}
 B_i & = & \inf\{t>A_i,\ X_t\notin J\},\\
  A_{i+1} & = & \min\{t>B_i,\ X_t \in J\},\\
  \alpha & = & \inf\{k \geq 0,\ X_{A_k}\in W_E\},\\
  \beta & =& \inf\{k>\alpha,\ X_{B_k} \notin R_{E,\gamma}\}.
 \end{eqnarray*}
 Since there is only a finite number of $C_j$, we can consider $\varepsilon_0$ and $c>0$ in Theorem \ref{TheoremsortieCoupe} such that if $\varepsilon_{t_0} \leq \varepsilon_0$, for all $i\geq 0$,
 \[\E\co \left.\int_{B_i}^{A_{i+1}} \po \varepsilon_{s}\pf^{-\frac12}e^{-\frac{E}{\varepsilon_{s}}} ds\right| B_i\cf \leq c. \]
 (Note that this would be false for $E=0$).
 
On the other hand, note that on each connected component of $J\smallsetminus W_E$, $U$ is monotone. There are two such components which are infinite, and where the potential is convex, so that the expected time the process stays there is bounded by a constant which only depends on $\varepsilon_0$. The time the process stays in a compact connected component of $J\smallsetminus W_E$ is bounded by twice the length of the component, and there is a finite number of such components. Thus there exist $c'$ such that if $\varepsilon_{t_0} \leq \varepsilon_0$, for all $i\geq 0$,
\[\E\co \left.\int_{A_i}^{B_i} \po \varepsilon_{s}\pf^{-\frac12}e^{-\frac{E}{\varepsilon_{s}}} ds\right| A_i\cf \leq c'. \]
Moreover as a consequence of part \eqref{PartieDeuxTheoremeSortie} of Theorem \ref{TheoremsortieCoupe}, $\E\po \alpha\pf$ is bounded by a constant which does not depend on the cooling schedule, and finally there exists $\tilde c>0$ such that
\[\E\co \int_{t_0}^{A_\alpha} \po \varepsilon_{s}\pf^{-\frac12}e^{-\frac{E}{\varepsilon_{s}}} ds\cf \leq \tilde c, \]
so that
\begin{eqnarray*}
 \mathbb P\po A_\alpha \geq r\pf & \leq  & \frac{\tilde c}{\int_{t_0}^{r} \po \varepsilon_{s}\pf^{-\frac12}e^{-\frac{E}{\varepsilon_{s}}} ds}.
\end{eqnarray*}
At time ${A_\alpha}$, the process attains the bottom of a cusp $C$ (and therefore is in $\bar A_C$) of depth $E+\gamma$ which is included in $R_{E,\gamma}$. Thanks to Theorem \ref{TheoremLowerBound},
\begin{eqnarray*}
 \mathbb P \po A_\beta \geq r\pf & \geq &  \exp\po - \Gamma \po \frac{e^{-\frac{E+\gamma}{\varepsilon_{t_0}}}}{\sqrt{\varepsilon_{t_0}}}+ \int_{t_0}^{r} \frac{e^{-\frac{E+\gamma}{\varepsilon_{u}}}}{\sqrt{\varepsilon_{u}}} du \pf\pf\\
  & \geq &  \exp\po - \Gamma e^{-\frac{\gamma}{\varepsilon_{t_0}}}\po\frac{e^{-\frac{E}{\varepsilon_{t_0}}}}{\sqrt{\varepsilon_{t_0}}}+ \int_{t_0}^{r} \frac{e^{-\frac{E}{\varepsilon_{u}}}}{\sqrt{\varepsilon_{u}}} du \pf\pf.
\end{eqnarray*}
Thus, for any $t_0$ and $r\geq t_0$,
\begin{eqnarray*}
 \mathbb P\po X_r \in R_{E,\gamma}\pf & \geq & \mathbb P \po A_\beta \geq r \geq A_\alpha\pf \\
 & \geq &  \exp\po - \Gamma e^{-\frac{\gamma}{\varepsilon_{t_0}}}\po\frac{e^{-\frac{E}{\varepsilon_{t_0}}}}{\sqrt{\varepsilon_{t_0}}}+ \int_{t_0}^{r} \frac{e^{-\frac{E}{\varepsilon_{u}}}}{\sqrt{\varepsilon_{u}}} du \pf\pf - \frac{\tilde c}{\int_{t_0}^{r} \po \varepsilon_{s}\pf^{-\frac12}e^{-\frac{E}{\varepsilon_{s}}} ds}.
\end{eqnarray*}
Let $h(t)$ be defined for any $t\geq 0$ by
\[ \int_{t}^{h(t)} \frac{e^{-\frac{E}{\varepsilon_{u}}}}{\sqrt{\varepsilon_{u}}} du = \frac{1}{\varepsilon_t}.\]
As a strictly increasing function it is invertible, and in particular $(t\rightarrow\infty) \Leftrightarrow \po h(t) \rightarrow \infty\pf$.
\begin{eqnarray*}
 \mathbb P\po X_{h(t)} \in R_{E,\gamma}\pf & \geq &  \exp\po - \Gamma e^{-\frac{\gamma}{\varepsilon_{t}}}\po \frac{e^{-\frac{E}{\varepsilon_{t}}}}{\sqrt{\varepsilon_t}}+ \frac{1}{\varepsilon_t} \pf\pf - \tilde c \varepsilon_t\\
 & \underset{h(t)\rightarrow\infty}\longrightarrow & 1.
\end{eqnarray*}
\end{proof}

\begin{proof}[Proof of Theorem \ref{TheoremNS}]
 We treat first the case of fast cooling, namely we assume that for all $\delta >0$,
  \[\int_0^\infty  \po \varepsilon_{s}\pf^{-\frac12}e^{-\frac{\delta}{\varepsilon_{s}}} ds < \infty.\]
 Let $x$ be a local minimum of $U$, $\delta >0$ and $C_\delta(x)$ be the set of points which are reachable from $x$ at height $U(x) + \delta$. Any neighbourhood of $x$ contains $C_\delta(x)$ for $\delta$ small enough. If at some time $t_0$ the process enters $C_\delta(x)$, Theorem \ref{TheoremLowerBound} yields
 \[\mathbb P \po X_t \in C_\delta(x)\ \forall t\geq t_0\pf \geq \exp\po - \Gamma \po \frac{e^{-\frac{\delta}{\varepsilon_{t_0}}}}{\sqrt {\varepsilon_{t_0}}}+ \int_{t_0}^{\infty} \frac{e^{-\frac{\delta}{\varepsilon_{u}}}}{\sqrt{\varepsilon_{u}}} du \pf\pf  >0.\]
 Thus, each time the process reaches a local minimum $x$, it has a positive probability to stay trapped forever in a neighbourhood of $x$. If it escapes, almost surely it will reach another local minimum later. Thus, the probability that it get trapped at some time is 1, and the probability that it's already been trapped at time $t$ goes to 1 as $t$ goes to infinity, so that, if $S$ is a neighborhood of all local minima of $U$,
 \[\underset{t\rightarrow\infty}\lim \mathbb P \po X_t \in S\pf = 1.\]
 Moreover, for any local minimum $x$, as the temperature is not allowed to vanish, there is a non-zero probability to reach $x$, and so to stay trapped in $C_\delta(x)$, which yields 
 \[\underset{t\rightarrow\infty}\liminf \mathbb P \po X_t \in C_\delta(x)\pf >0.\]
 Thus, we have proved the part 1 and the ``if'' ($\Rightarrow$) half of part 2 of Theorem \ref{TheoremNS} in the case of fast cooling. Finally, the ``only if'' half of part 2 is tautological in this case.
 
 Concerning the part 3, recall that we have supposed there is at least one non-global minimum $\tilde x$, near which the process has a non-zero probability to stay forever. In this event $U(X_t) \geq U(\tilde x) > \underset{\R}\min U + \delta $ for $\delta$ small enough, so that
 \[\underset{t\rightarrow\infty}\lim \mathbb P\po U(X_t) < \underset{\R}\min U + \delta\pf \hspace{15pt} \leq \hspace{15pt} 1 - \mathbb P \po X_t\text{ gets trapped near }\tilde x\pf \hspace{15pt} < \hspace{15pt} 1.\]
 
 \bigskip 
 
 Now we turn to slow cooling, namely we suppose there exists $F>0$ such that $\forall \delta \neq F$,
  \[\po \int_0^\infty  \po \varepsilon_{s}\pf^{-\frac12}e^{-\frac{\delta}{\varepsilon_{s}}} ds = \infty\pf \hspace{20pt} \Leftrightarrow \hspace{20pt}  \po \delta < F \pf.\]
  Then, for all $\delta < F$, according to Lemma \ref{LemPegalUn},
 \[\underset{t\rightarrow\infty}\lim \mathbb P \po X_t \in R_{\delta}\pf = 1.\]
  Since any neighbourhood of all local minima contains $R_\delta$ for $\delta$ small enough, part 1 of Theorem \ref{TheoremNS} is proved.
  
  Let $E>0$ be such that
  \[ \int_0^\infty  \po \varepsilon_{s}\pf^{-\frac12}e^{-\frac{E}{\varepsilon_{s}}} ds = \infty,\]
 and $S$ be a neighbourhood of all minima of depth $E$ such that $S^c$ is a neighbourhood of all other minima. We want to prove 
 \[\underset{t\rightarrow\infty}\lim \mathbb P \po X_t \in S\pf = 0.\]
 Lemma \ref{LemPegalUn} together with the first part of Theorem \ref{TheoremNS} we have just proven implies that for any neighbourhood $\mathcal U$ of the minima in $R_E$,
 \[\underset{t\rightarrow\infty}\lim \mathbb P \po X_t \in \mathcal U\pf = 1.\]
 Thus, since $S^c$ is a neighbourhood of all the minima it contains, it is enough to prove that $R_E$ does not contain any minimum of depth exactly $E$.
 
 Let $x$ be such a minimum, and let $z$ be such that $U(z) < U(x)$ and $z$ is reachable from $x$ at height $U(x)+E$. Suppose $x\in R_E$, and let $y\in W_E$ be such that $x$ is reachable from $y$ at height $U(y)+E$.
 
 If $U(y) < U(x)$, since $y$ is reachable from $x$ at depth $U(y)+E < U(x) + E$, by definition of the depth of a local minimum, it means $x$ is of depth strictly less than $E$, which is a contradiction.
 
 On the other hand if $U(y) \geq U(x)$, then $z$ is reachable from $y$ at height $U(y)+E$, while $U(z) <U(y)$ , which is contradictory with the fact that $y$ is of depth strictly larger than $E$. 
 
 This means $R_E$ does not contain any local minimum of depth $E$.
 
 \bigskip
 
 At this point we have proven part 1 and implication $\Leftarrow$ of part 2 of Theorem \ref{TheoremNS}, which we will use to prove the converse.
 
 \bigskip
 
 Now suppose 
  \[ \int_0^\infty  \po \varepsilon_{s}\pf^{-\frac12}e^{-\frac{E}{\varepsilon_{s}}} ds < \infty.\]
 In particular $E\geq F$. Let  $x$ be a minimum of depth $E$. As a first step, assume
  \[ \int_0^\infty  \po \varepsilon_{s}\pf^{-\frac12}e^{-\frac{F}{\varepsilon_{s}}} ds < \infty.\]
Let $C$ be the set of all points which are reachable from $x$ at height strictly less than $U(x)+F$  (if $F=E$, $C=C_x$). Then $C$ is a cusp of depth $F$, whose bottom $B$ is constituted of minima of depth exactly $E$, and the depth of any other minimum in $C$ is strictly less than $F$ (since $B$ is reachable from them without leaving $C$).
 
 From Theorem \ref{TheoremLowerBound}, the process has a non-zero probability to stay trapped forever in $C$, so that
 \[\underset{t\rightarrow\infty}\liminf\ \mathbb P \po X_t \in C\pf >0.\]
 On the other hand, if $S\subset C$ is a neighbourhood of $B$, since the depth $d$ of any minimum in $C\smallsetminus S$ satisfies $d<F$ and so
  \[ \int_0^\infty  \po \varepsilon_{s}\pf^{-\frac12}e^{-\frac{d}{\varepsilon_{s}}} ds = \infty,\]
  from part 1 and implication $\Leftarrow$ of part 2 of Theorem \ref{TheoremNS}
 \[\underset{t\rightarrow\infty}\lim\ \mathbb P \po X_t \in C\smallsetminus S\pf = 0.\]
Hence,
 \[\underset{t\rightarrow\infty}\liminf\ \mathbb P \po X_t \in S\pf >0\]
 which ends the proof of part 2.
 
 A slight adaptation is needed when
  \[ \int_0^\infty  \po \varepsilon_{s}\pf^{-\frac12}e^{-\frac{F}{\varepsilon_{s}}} ds = \infty.\]
  In this case, necessarily $F<E$. Let $\eta \in(0,F-E)$ and $C$ be the set of all points which are reachable from $x$ at height $U(x)+F+ \eta$. Then $C$ is a cusp of depth $F+\eta$ whose bottom $B$ is constituted of minima of depth exactly $E$, and whose other minima are all of depth strictly less than $F+\eta$ (since $B$ is reachable from them without leaving $C$). Since there is only a finite number of such minima, in fact if $\eta$ is small enough these non-global minima in $C$ are even of depth less than $F$ (possibly equal). 
  
  Since $F+\eta > F$, the process has a non-zero probability to stay trapped forever in $C$. On the other hand, since all non-global minima in $C$ are of depth less than $F$, 
 \[\underset{t\rightarrow\infty}\lim\ \mathbb P \po X_t \in C\smallsetminus S\pf = 0.\]
 as soon as $S$ is a neighbourhood of $B$. Thus the same conclusion holds.
 
 \bigskip
 
 As in the fast cooling case, part 3 is a direct consequence of parts 1 and 2 and of the presence of at least one non-global minimum.
  
\end{proof}

\section{Non-minimal rate}\label{SectionNonMinimal}

As have been seen in Section \ref{SectionDefRTP}, the measure $\mu_\varepsilon = e^{-\frac{U(x)}{\varepsilon}}dx \otimes \frac{\delta_1+\delta_{-1}}{2}$ is invariant for the Markov process with generator
\begin{eqnarray}\label{DefGeneResiduel}
Lf(x,y) & = & y f'(x,y) + \lambda(x,y) \po f(x,-y) - f(x,y)\pf
\end{eqnarray}
if and only if $\lambda(x,y) = \frac{\po yU'(x)\pf_+}\varepsilon + r(x)$, where $r$ can be any non-negative function. A positive $r$ is a residual rate of jump which brings randomness in the system at any time. The previous works on explicit estimations of convergence to equilibrium for the velocity jump process (\cite{Fontbona2010,Calvez,Monmarche2013}) all assume $r$ is bounded below by a positive constant $r_*$. 

If $r$ is bounded from above by a constant $r^*$ uniformly in $x$ and $\varepsilon$, it is expected that he behaviour of the
process does not change. Indeed, at low temperature, it only adds random jumps at exponential (macroscopic, in the sense: of order of magnitude independent from $\varepsilon$) times to the minimal-rate dynamics, and the latter accounts both for the way the process overcome an energy barrier, and the metastable behaviour in the vicinity of a local minimum, namely: many microscopic excursions of length of order $\sqrt \varepsilon$. In the rest of this section, we will make this statement more precise, and prove it.

We will consider the context of Section \ref{SectionTemperatureConstante}, that is a uni-dimensional double-well potential $U$ with its three local extrema $x_0<x_1<x_2$ and an homogeneous Markov process starting at  $(X_0,Y_0) = (x_0,1)$; but now the generator is given by \eqref{DefGeneResiduel} with $r\neq 0$. We are interested in the following:
\begin{eqnarray}
\eta & = & \inf\left\{t>0,\ X_t \in \{x_0,x_1\} \right\},\label{EqEta}\\
p_{x_0} & = & \mathbb P\po X_{\eta} = x_1\pf.\notag
\end{eqnarray}

\begin{prop}
Whatever the residual rate of jump $x\mapsto r(x)$,
\begin{eqnarray*}
p_{x_0} & = & \frac{e^{-\frac{ U(x_1)-U(x_0)}\varepsilon}}{1+\int_{x_0}^{x_1} r(z)e^{-\frac{ U(x_1)-U(z)}\varepsilon}dz }
\end{eqnarray*}
\end{prop}
\textbf{Remarks :}\begin{itemize}
\item A positive residual rate of jump can only worsen the probability to overcome an energy barrier in one shot.
\item A positive residual rate of jump in the neighbourhood of $x_1$ does more harm than the same rate near $x_0$.
\end{itemize}
\begin{proof}
If $x\in(x_0,x_1)$, let $p_x$ be the probability that the process $(X_t,Y_t)_{t\geq0}$, starting from $(x,1)$, reaches $(x_1,1)$ before $(x,-1)$.
Suppose the process starts at $(x-s,1)$ for some small $s>0$. The probability the process goes from $x-s$ to $x$ without any jump is $1-s\po \frac{U'(x)}{\varepsilon}+r(x)\pf + \underset{s\rightarrow 0}o(s)$. The probability it reaches $(x,1)$ before $(x-s,-1)$ but with at least one jump (and so with at least two jumps) is of order $s^2$ (when $s\rightarrow 0$). Once the process has reached $(x,1)$, it has a probability $p_x$ to reach $(x_1,1)$ before having fallen back to $(x,-1)$. If nevertheless it has fallen back to $(x,-1)$ - which occurs with probability $(1-p_x)$ - it has a probability $sr(x) + \underset{s\rightarrow 0}o(s)$ to jump before reaching $(x-s,-1)$, in which case it reaches again $(x,1)$ with probability $1+\underset{s\rightarrow 0}o(1)$. In this latter case, it reaches $(x_1,1)$ before $(x-s,-1)$ with probability $p_x + \underset{s\rightarrow 0}o(1)$. Thus everything boils down to
\begin{eqnarray*}
p_{x-s} & = & \po 1-s\po \frac{U'(x)}{\varepsilon}+r(x)\pf\pf \po p_x + s(1-p_x)p_xr(x)\pf + \underset{s\rightarrow 0}o(s)\\
 & = & p_x -s p_x\po \frac{U'(x)}{\varepsilon} + p_x r(x)\pf  + \underset{s\rightarrow 0}o(s).
\end{eqnarray*}
We shall solve the differential equation, for all $x\in(x_0,x_1)$,
\[\partial_x p_x = p_x\po \frac{U'(x)}{\varepsilon} + p_x r(x)\pf,\]
by looking for a solution of the form $p_x = h(x) e^{-\frac{U(x_1)-U(x)}{\varepsilon}}$. This yields
\begin{eqnarray*}
\po \frac{1}{h}\pf'(x) & = & - r(x) e^{-\frac{U(x_1)-U(x)}{\varepsilon}},
\end{eqnarray*}
and we conclude with $h(x_1) = p_{x_1} = 1$.
\end{proof}
To study $\eta$, the length of an excursion away from $x_0$, we construct the process with generator \eqref{DefGeneResiduel} in the following way: let $(E_k)_{k\geq1}$ and $(F_k)_{k\geq1}$ be two independent sequences of independent r.v. with law $\mathcal E(1)$. 
 Let $V_1$ and $W_1$ be defined by
\[\varepsilon E_1 = \int_0^{V_1} \po Y_{0} U'(X_{0} + Y_{0} s)\pf_+ ds,\hspace{25pt}F_1 = \int_0^{W_1} r(X_{0} + Y_{0} s) ds.\]
Let $S_1 = \min\po V_1,W_1\pf$. From $T_0 = 0$ to $T_1 = S_1$ the process goes on deterministically, in the sense that $Y_{s} = Y_{0}$ and $X_{s} = X_{0} + s Y_{0}$. At time $T_1$ the process jumps, so that $Y_{T_1} = - Y_{0}$. If $S_1 = V_1$, the process jumps due to the minimal rate; we will call this a first type jump. If $S_1 = W_1$ the jump is due to the residual rate and we will call this a second type jump. When the process has been defined up to a jump time $T_j$, we start the same procedure again, replacing $E_1$ by $E_{j+1}$ and $F_1$ by $F_{j+1}$.


Since $U$ is supposed smooth, $U^{(3)}(x)$ is bounded on $[x_0,x_1]$, interval on which $U$ is strictly increasing. Thus it exists $\rho >0$ such that for all $x\in[x_0,x_1)$, for all positive $s \leq  x_1-x$,
\[ U(x+s) \geq U(x) + \rho s^2.\]
We define
\[D_k = \min\po |x_1 - x_0|,\ \sqrt{\frac{\varepsilon E_{2k-1}}{\rho}}\pf,\]
which is independent from the sequence $(F_k)_{k\geq 0}$.
On the other hand, we consider the events
\[A_k = \left\{ F_k < \int_0^{x_1-x_0} r(z)dz\right\} \]
and the geometric variable
\[N = \max\left\{k\geq 1,  A_{2k}\right \},\]
which is independent from the sequence $(E_k)_{k\geq 0}$.
\begin{lem}
Whatever the residual rate of jump $x\mapsto r(x)$,
\begin{eqnarray*}
\eta & \leq & 2\sum_{k=1}^N D_k,
\end{eqnarray*}
where $\eta$ is defined by \eqref{EqEta}.
\end{lem}
\begin{proof}
From time 0 to $\eta$, the process alternates between ascending and downward moves. If $X_\eta = x_0$, the process spends as much time going up as going down, and if $X_\eta = x_1$ then it spends more time going up than down; in either case, $\eta$ is less than twice the cumulated ascending time. Let $n$ be the number of ascending moves before the time $\eta$. At the end of the $k^{th}$ ascending move, the next time of jump will be defined thanks to the variable $F_{2k}$ (as long as the process does not reach $(x_0,-1)$). In the event $\bar A_{2k}$ (namely the negation of $A_{2k}$) there is no second type jump before the process reaches $(x_0,-1)$. This means $n \leq N$.

Let $(d_k)_{1 \leq k \leq n}$ be the duration of the $n$ ascending moves, and $(Z_k)_{1 \leq k \leq n}$ be the starting point of these moves (for instance $Z_1 = x_0$). 
\begin{eqnarray*}
d_k & \leq & V_{2k-1}\\
& = & \sup\left\{t<(x_1-x),\  \varepsilon E_{2k-1} > U(Z_k+t)-U(Z_k) \right\}\\
& \leq & \sup\left\{t<(x_1-x_0),\  \varepsilon E_{2k-1} > \rho t^2 \right\}\\
& = & D_k.
\end{eqnarray*}
Finally,
\[\eta \leq 2\sum_{k=1}^n d_k \leq 2\sum_{k=1}^N D_k.\]
\end{proof}
Now if $r$ is supposed bounded above by a constant $r^*$,
\[\E\co N\cf = 1 + \frac{1}{\mathbb P(\bar A_2)} \leq 1 + e^{(x_1-x_0)r^*},\]
and thus, $N$ and the $D_k$'s being independent, $\E\co\eta\cf \leq c\sqrt{\varepsilon}$ for some constant $c$. On the other hand,
\[\E\co \eta \cf \geq \mathbb P\po \bar A_1,\bar A_2\pf \E\co \eta\ |\ \bar A_1,\bar A_2\cf \geq e^{-2(x_1-x_0)r^*} \E\co \min( V_1, x_1-x_0)\cf.\]
From Lemma \ref{LemRacineEpsi} we get that $\E\co \eta\cf \geq c' \sqrt\varepsilon$ for some $c'>0$. But if $\E\co \eta \cf$ is of order $\sqrt\varepsilon$, it means when $\varepsilon$ goes to 0 it is unlikely that a second type jump occurs during an excursion, so that only the asymptotic of $\E\co V_1\cf$ should intervene. Indeed, we can prove the following:
\begin{prop}
If $r \leq r^*$,
\begin{eqnarray*}
\E\co \eta \cf  & = & \sqrt{\frac{2\pi \varepsilon}{ U''(x_0)}}   \po 1 + \underset{\varepsilon \rightarrow 0}o (1)\pf
\end{eqnarray*}
\end{prop}
\begin{proof}
When $r\leq r^*$, $F_k \leq r^* W_k$. Since  $V_k \leq D_k$, the event $B = \{r^*D_1 \leq \min(F_1,F_2,\delta)\}$ (for some $\delta >0$) implies that $V_1 \leq \min(W_1,W_2)$: there is no second type jump during the excursion. 
We decompose
\begin{eqnarray*}
\E\co\eta\cf & = \E\co \eta  \mathbb 1_{B}\cf +  \E\co \eta \mathbb 1_{\bar B}\cf.
\end{eqnarray*}
First, note that $\mathbb P\po B\pf$ goes to 1 when $\varepsilon$ goes to 0. Second,
\begin{eqnarray*}
\E\co\mathbb 1_{\bar B} \eta \cf & \leq & 2\E\co \mathbb 1_{\bar B}\sum_{k=1}^N D_k \cf\\
& = & 2\E\co \mathbb 1_{\bar B} D_1\cf + 2 \mathbb P\po \bar B\pf \E\co N-1\cf \E\co D_2\cf\\
& \leq & 2\sqrt{\mathbb P\po \bar B\pf \E\co  D_1^2 \cf} + 2 \mathbb P\po \bar B\pf \E\co N-1\cf \E\co D_2\cf\\
& = & \underset{\varepsilon\rightarrow 0}o\po\sqrt \varepsilon\pf.
\end{eqnarray*}
Suppose that $B$ holds and  $\delta < r^*|x_1-x_0|$. Then the process goes up until a first time jump occurs, then it falls back to $x_0$, and no second type jump occurs: $\eta = 2V_1$, and
\begin{eqnarray*}
\varepsilon E_1 & = & U(x_0 + V_1) - U(x_0)\\
& = & \frac12 U''(x_0)\po \frac{V_1}{\sqrt \varepsilon}\pf^2 \varepsilon + \underset{\varepsilon\rightarrow 0}o\po  \varepsilon\pf\\
\Rightarrow\hspace{15pt}\frac{V_1}{\sqrt \varepsilon} & \underset{\varepsilon\rightarrow 0}\longrightarrow & \sqrt{\frac{2 E_1}{U''(x_0)}}\hspace{10pt}a.s.
\end{eqnarray*}
Moreover $\mathbb 1_{\bar B} \rightarrow \mathbb 1$ almost surely, and
\[\frac{V_1}{\sqrt \varepsilon} \mathbb 1_{\bar B} \leq \frac{D_1}{\sqrt \varepsilon}  \leq \sqrt{\frac{E_1}{\rho}}.\]
By the dominated convergence theorem,
\begin{eqnarray*}
2\E\co V_1\mathbb 1_{\bar B} \cf & = & 2\sqrt\varepsilon \E\co \sqrt{\frac{2 E_1}{U''(x_0)}}\cf \po 1 + \underset{\varepsilon \rightarrow 0}o (1)\pf,
\end{eqnarray*}
which concludes this proof.
\end{proof}
As a conclusion, note than when $r\leq r^*$,
\[e^{-\frac{U(x_1)-U(x_0)}{\varepsilon}} \geq p_{x_0} \geq \frac{e^{-\frac{U(x_1)-U(x_0)}{\varepsilon}}}{1+(x_1-x_0)r^*}.\]
From these bounds, following the proof of Theorem \ref{TheoremNS}, we can see the same result holds even if the non-minimal rate is not zero (as long as it is bounded above uniformly in time). On the other hand if $r$ goes to infinity as $\varepsilon$ goes to 0, the velocities at two different instants are more and more decorrelated, and we may think the (suitably rescaled) process ends up with a genuine diffusive part. We recall this has been proved at least for the case of a constant potential on the torus, in \cite{Volte-Face}.

\section{The multidimensional process}\label{SecDimd}

In this section we consider an inhomogeneous process $Z^{\varepsilon}$ on $\mathbb T^d\times\Sd$ with generator $(L_t)_{t\geq 0}$ given by \eqref{EqGenerateurDimD}, with a positive cooling schedule $(\varepsilon_t)_{t\geq0}$, and its associated semi-group $(P_{s,t})_{t\geq s\geq 0}$. To lighten the notations, in this section we drop the $\varepsilon$ exponent and only write $Z=(X,Y)$. In order to prove Theorems \ref{ThmDimDConstant} and \ref{ThmDimDRecuit} we will establish the following intermediary result:

\begin{thm}\label{ThmDimdCouple}
There exist $c,\theta>0$ that depend only on $d,r,U$  such that for all cooling schedule, all initial law $\mu$, and all $t\geq 0$,
\begin{eqnarray*}
\|\mu P_{t,t+\sqrt d}  - \nu_{t}\|_{TV} & \leq & \po 1 - ce^{-\frac{\theta }{\varepsilon_t}}\pf \|   \nu - \nu_{t}\|_{TV} + \theta\po \frac1{\varepsilon_{t+\sqrt d}}-\frac1{\varepsilon_{t}}\pf.
\end{eqnarray*}
\end{thm}


In the first instance we consider the case $U=0$:

\begin{lem}\label{LemDimDU0}
Let $\po P^0_t\pf_{t\geq 0}$ be the semi-group on $\mathcal P\po\mathbb T^d\times \Sd\pf$ generated by
\begin{eqnarray}\label{EqGeneDimDU0}
L f(x,y) &=& y.\n_x f(x,y)   + r \po \int f(x,z) dz - f(x,y)\pf.
\end{eqnarray}
 Then, there exists $c>0$ which depends only on $d$ and $r$ such that for all $\mu_1,\mu_2\in\mathcal P\po\mathbb T^d\times \Sd\pf$,
\begin{eqnarray*}
\left\|\mu_1   P^0_{\sqrt d} - \mu_2 P^0_{\sqrt d} \right\|_{TV} & \leq & (1-c) \left\| \mu_1-\mu_2\right\|_{TV} 
\end{eqnarray*}
\end{lem}
\begin{proof}
Let $(Z_0,Z_0')$ be an optimal coupling of $\mu_1$ and $\mu_2$, in the sense that the law of $Z_0$ (resp. $Z_0'$) is $\mu_1$ (resp. $\mu_2$) and that 
\[\|   \mu_1 - \mu_2\|_{TV} =  \mathbb P\po Z_0\neq Z_0'\pf.\]
We aim to construct a coupling $(Z_t,Z_t')_{t\in\co 0, \sqrt d\cf}$ for which each marginal is a Markov process with generator \eqref{EqGeneDimDU0} and such that $Z_{\sqrt d}$ and $Z_{\sqrt d}'$ are likely to be equal. If $Z_0=Z_0'$ we define $Z$ up to time $\sqrt d$ and then let $Z_t' = Z_t$ for all $t\in(0,\sqrt d]$. 

Now if $Z_0 \neq Z_0'$ we will show that  the (density part of the) laws of $Z_{\sqrt d}$  and $Z_{\sqrt d}'$ are bounded below by a constant time the uniform density on $\mathbb T^d \times \mathbb S^{d-1}$, which yields a lower bound on their total variation distance. Indeed, let $(E_i)_{i\geq 0}$ be a sequence of i.i.d. exponential variables and defined $S_0 =0$ and $S_{i+1} = S_i + \frac1r E_i$ to be the jump times of $Z$. Under the event $\{E_3> r\sqrt d\}$ there are at most two jumps before time $\sqrt d$. More precisely consider the event
\[B=\left\{E_3> r\sqrt d,\ \frac{r \sqrt d}2 < E_1+E_2 < r\sqrt d\right\}\]
 so that if $f$ is any positive function, starting from $Z_0=(x_0,y_0)$,
\begin{eqnarray*}
\mathbb E\po f(Z_{\sqrt d}) \pf & \geq & \mathbb E\po f(Z_{\sqrt d}) \mathbb 1_{B} \pf\\
& \geq & \mathbb P\po E_3> r\sqrt d \pf  \int_{t=\frac12\sqrt d}^{\sqrt d} \int_{s=0}^{t}\int  f\po\tilde x , y_2\pf    r^2 e^{-rt}d y_1 d y_2 d sd t,
\end{eqnarray*}
where $dy_1$ and $dy_2$ stand for the uniform law on $\Sd$ and
\[\tilde x =  x_0 + s y_0 + (t-s) y_1 + (\sqrt d - t) y_2.\]
Suppose $t>\frac12\sqrt d$ and $y_2\in\Sd$ (and thus $x_0' := x_0 +  (\sqrt d - t) y_2$) are fixed. We want to prove
\begin{eqnarray}\label{Eqcouplelem}
 \int_{0}^{t}\int_{\Sd} f\po\tilde x , y_2\pf    e^{-{rs}} d s d y_1 
& \geq & c_1 \int_{\mathbb T^d } f\po u , y_2\pf   d u 
\end{eqnarray}
for some $c_1>0$. Let $u\in\mathbb T^d$ and $\delta>0$, $S$ be a r.v. uniform on $[0,1]$ and $Y$ be uniformly distributed on $\Sd$ and independent from $S$, so that \eqref{Eqcouplelem} reads
\[\mathbb P\po | x_0' + tS y_0 + t(1-S) Y - u |\leq \delta\pf \geq c_2 \delta^{d}\]
with a constant $c_2$ that depends neither on $x_0'$, $t$, $u$, $\delta$ nor $y_0$. Since the diameter of $\mathbb T^d$ is $\frac12\sqrt d<t$, there exists $s_*\in(0,1)$ and $y_*\in\Sd$ such that  $u = x_0'+ts_*y_0+t(1-s_*)y_*$.
\begin{eqnarray*}
\mathbb P\po | x_0' + tS y_0 + t(1-S) Y - u |\leq \delta\pf & \geq & \mathbb P\po |S-s_*| + |Y-y_*| \leq \frac{\delta}{2\sqrt d}\pf\\
& \geq & \mathbb P\po |S-s_*| \wedge |Y-y_*| \leq \frac{\delta}{4\sqrt d}\pf\\
&\geq & c_2 \delta^{d}
\end{eqnarray*}
for some $c_2>0$ since $S$ and $Y$ are uniformly drawn on compact spaces with respective dimensions 1 and $d-1$. Bringing the pieces together we have obtained
\begin{eqnarray*}
\mathbb E\po f(Z_{\sqrt d}) \pf & \geq & r^2 e^{-2r\sqrt d} c_1 \int_{\mathbb T^d\times \Sd } f\po u , y_2\pf   d u dy_2
\end{eqnarray*}
for any $f\geq 0$. In other words for any Borel set $\mathcal A \subset \mathbb T^d\times \Sd$,
\begin{eqnarray*}
\mathbb P\po  Z_{\sqrt d} \in \mathcal A \pf & \geq & c \mathbb P\po  H \in \mathcal A \pf
\end{eqnarray*}
where $c=r^2 e^{-2r\sqrt d} c_1$ and $H$ is a r.v. with uniform law $dx\times dy$  on $\mathbb T^d\times \Sd$. This means there exists a probability measure $m$ (that depends on $Z_0$) such that the law of $Z_{\sqrt d}$ is $ c (dx\times dy) + (1-c)m$. The same goes for $Z'$ with a law $m'$. With probability $c$, set $Z_{\sqrt d}= H = Z'_{\sqrt d}$, and with probability $(1-c)$, draw $Z_{\sqrt d}$ (resp. $Z'$) according to the law $m$ (resp. $m'$). with this coupling,
\begin{eqnarray*}
\mathbb P\po Z_{\sqrt d} \neq Z_{\sqrt d} '\pf &\leq& (1-c) \mathbb P\po Z_{0} \neq Z_{0} '\pf .
\end{eqnarray*}
\end{proof}

\begin{proof}[Proof of Theorem \ref{ThmDimdCouple}]
Let $(Z_0,Z_0')$ be an optimal coupling of $\mu$ and $\nu_t$, in the sense that the law of $Z_0$ (resp. $Z_0'$) is $\mu$ (resp. $\nu_t$) and that 
\[\|   \mu - \nu_{{t}}\|_{TV} =  \mathbb P\po Z_0\neq Z_0'\pf.\]
If $Z_0=Z_0'$ let $Z'=(X',Y')$ evolve as a Markov process with generator $L_{t}$ (here $t$ is fixed). Since $\nu_{ {t}}$ is invariant for $L_{t}$, $Z_{\sqrt d}' \sim \nu_{ {t}}$. Let $E$ be an exponential r.v. independent from $(Z'_s)_{s\in[0,\sqrt d]}$ and
\begin{eqnarray*}
S & = & \inf \left\{s\geq t,\ E < \int_{t}^s \po Y_{u}' \nabla_x U(X_{u}')\pf_+ \po \frac{1}{\varepsilon_{u}} - \frac{1}{\varepsilon_{t}} \pf du\right\}.
\end{eqnarray*}
Define $Z_s = Z_s'$ up to time $T=S\wedge \sqrt d$.  If $T=S$, the process $Z$ is reflected at time $T$ (in the sense $Y\leftarrow Y^*$) and then we let $Z$ evolve as an inhomogeneous Markov process with generator $\po L_{s}\pf_{s\geq 0}$ independently from $Z'$. Note that  $T=\sqrt d$ implies $Z_{\sqrt d}=Z'_{\sqrt d}$, and that 
\begin{eqnarray}\label{EqPTS}
\mathbb P\po T <S\pf & \geq & \mathbb P\po E > \sqrt d \| \nabla_x U\|_\infty\po \frac{1}{\varepsilon_{t+\sqrt d}} - \frac{1}{\varepsilon_{t}} \pf\pf.
\end{eqnarray}
This means
\begin{eqnarray}\label{EqContractTV}
\lefteqn{\| \mu P_{t,t+\sqrt d} - \nu_{t}\|_{TV}  \leq }\notag\\
& &  \mathcal P\po Z_{\sqrt d} \neq Z_{\sqrt d}'\ |\ Z_0 \neq Z_0',T<S\pf \left\|\mu-\nu_{ {t}}\right\|_{TV}  + \po 1 -e^{-\sqrt d \| \nabla_x U\|_\infty \po \frac{1}{\varepsilon_{t+\sqrt d}} - \frac{1}{\varepsilon_{t}} \pf}\pf\notag\\
&\leq &  \mathcal P\po Z_{\sqrt d} \neq Z_{\sqrt d}'\ |\ Z_0 \neq Z_0',T<S\pf \left\|\mu-\nu_{ {t}}\right\|_{TV}  +  \sqrt d \| \nabla_x U\|_\infty \po \frac{1}{\varepsilon_{t+\sqrt d}} - \frac{1}{\varepsilon_{t}} \pf .
\end{eqnarray}
In the following, suppose $T<S$ and $Z_0\neq Z_0'$. Let $W=(R,Q)$ and $W'=(R',Q')$ be Markov processes associated to the generator \eqref{EqGeneDimDU0} with respectively $W_0 = Z_0$ and $W_0'=Z_0'$ and such that 
\[\mathbb P\po W_{\sqrt d}\neq W_{\sqrt d}'\pf = \| \mu P_{\sqrt d}^0 -\nu_{t} P_{\sqrt d}^0  \|_{TV}.\]
Let $F$ be an exponential r.v. which is independent from all the other variables and
\begin{eqnarray*}
S_1 & = & \inf \left\{s\geq t_0,\ \varepsilon_{t} F <   \int_{t}^s \po Q_{u} \nabla_x U(R_{u})\pf_+  du \right\}\\
S_2 & = & \inf \left\{s\geq t_0,\ \varepsilon_{t} F <   \int_{t}^s \po Q'_{u} \nabla_x U(R'_{u})\pf_+  du \right\}
\end{eqnarray*}
Define $Z=W$ up to time $T_1 = S_1 \wedge \sqrt d$ and $Z'=W'$ up to time $S_2\wedge \sqrt d$. If $T_1=S_1$ the process $Z$ is reflected at time $T_1$ and then let it evolve independently from $W$, and similarly for $Z'$. Note that the event $\{\varepsilon_{t} F > \kappa\}$ is independent from $\{ W_{\sqrt d} =W_{\sqrt d}'\}$, so that
\begin{eqnarray}\label{EqFd}
\mathcal P\po Z_{\sqrt d} = Z_{\sqrt d}'\ |\ Z_0 \geq Z_0',T<S\pf & \geq & \mathbb P\po\varepsilon_{t} F > \sqrt d \| \nabla_x U\|_\infty \pf \mathbb P\po W_{\sqrt d} =W_{\sqrt d}' \pf\\
&\geq & c e^{-\frac{\sqrt d \| \nabla_x U\|_\infty}{\varepsilon_{t}}} \notag
\end{eqnarray}
where $c$ is given by Lemma \ref{LemDimDU0}.
\end{proof}
\textbf{Remarks:}
\begin{itemize}
\item Concerning the constants involved in Theorem \ref{ThmDimdCouple}, with the previous proofs we obtained that it is possible to choose $c \leq r^2e^{-2\sqrt d}c_1$, where $c_1$ only depends on $d$, and $\theta \leq \sqrt d \| \nabla_x U\|_\infty$. Nevertheless in the proof of Lemma \ref{LemDimDU0}, notice that the bound $c$ on the probability of success for the coupling is obtained by considering only  the situation where both processes jump twice and only twice. Hence, the bound \eqref{EqPTS} may be improved to
\begin{eqnarray*}
\mathbb P\po T <S\pf & \geq & \mathbb P\po E > \kappa(\sqrt d)\po \frac{1}{\varepsilon_{t+\sqrt d}} - \frac{1}{\varepsilon_{t}} \pf\pf,
\end{eqnarray*}
where we define $\kappa(t)$ as the supremum of $\int_0^t \po q(u)\nabla_x U\po r(u)\pf\pf_+ du$ over all the $(r,q)$ which are the trajectory of a velocity jump process that jumps only twice during time $t$. The same goes at line \eqref{EqFd} and thus Theorem \ref{ThmDimdCouple} holds with $\theta \leq \kappa_2(\sqrt d)$.
 
\item  when we proved \eqref{EqContractTV} we hadn't use the fact $\frac12\sqrt d$ is the diameter of $\mathbb T^d$ yet, which means that in fact for all $s\leq \sqrt d$
\begin{eqnarray*}\label{EqContractTVs}
 \|\mu P_{t,t+s}  - \nu_{t}\|_{TV}  &\leq 
&    \left\|\mu-\nu_{ {t}}\right\|_{TV}  + \sqrt d \| \nabla U\|_\infty \po \frac{1}{\varepsilon_{t+s}} - \frac{1}{\varepsilon_{t}} \pf .
\end{eqnarray*}
\end{itemize}

\begin{proof}[Proof of Theorem \ref{ThmDimDConstant}]
As a direct consequence of the previous remark and of Theorem \ref{ThmDimdCouple} we obtain, if $\varepsilon_t = \varepsilon$ is constant (so that $P_{0,t}=P_t$ is homogeneous),
\begin{eqnarray*}
 \| \mu P_{t } - \nu_{\infty}\|_{TV}  &\leq &   e^{-ce^{-\frac{\theta}{\varepsilon}} \left\lfloor\frac{t}{\sqrt d}\right\rfloor} \left\|\mu-\nu_{ \infty}\right\|_{TV} \\
 &\leq &   e^{-\frac{c}{\sqrt d}e^{-\frac{\theta}{\varepsilon}} \po t- \sqrt d\pf} \left\|\mu-\nu_{ \infty}\right\|_{TV}.
\end{eqnarray*}
\end{proof}

\begin{proof}[Proof of Theorem \ref{ThmDimDRecuit}]
Let $t_0$ be such that $\partial_s \po \frac{1}{\varepsilon_s}\pf \leq \frac{1}{(\theta+\eta) s}$ holds for $s\geq t_0$, let $t>2 t_0  + \sqrt d$ and $n = \left\lfloor \frac{t-t_0}{\sqrt d}\right\rfloor$. For $k\in\cco 0,n\ccf$ let $t_{n-k}= t-k\sqrt d$,
\begin{eqnarray*}
d_k &=&\|\mu P_{0,t_k} - \nu_{t_k}\|_{TV}, \\
a_k &=&    c e^{-\frac{\theta}{\varepsilon_{t_k}}},\\
b_k &=& \theta \po \frac{1}{\varepsilon_{t_{k+1}}} - \frac{1}{\varepsilon_{t_k}}\pf,
\end{eqnarray*}
so that Theorem \ref{ThmDimdCouple} reads $d_{k+1}-d_k \leq -a_k d_k + b_k$. This situation is a discrete analogous of \cite[Lemme 6]{Miclo92}. 
 From $\po\frac 1{\varepsilon_t}\pf' \leq \frac{1}{(\theta+\eta) t}$,
 \[ b_k \leq \frac{\theta \sqrt d }{(\theta+\eta)t_{k}}, \hspace{40pt} a_k \geq m \po \frac{1}{t_k}\pf^{\frac{\theta}{\theta+\eta}},\hspace{40pt} c_k:=\frac{b_k}{ak}\leq m' \po \frac{1}{t_k}\pf^{\frac{\eta}{\theta+\eta}}\]
for some $m,m'>0$. 
 For a fixed $l\in\mathbb N$ and for all $k\geq l $, $f_k:=d_k - m' t_{l}^{-\frac{\eta}{\theta+\eta}}$ satisfies
 \begin{eqnarray*}
 f_{k+1} - f_k & \leq & -a_k f_f\\
 \Rightarrow\hspace{20pt} f_k & \leq & f_l \prod_{j=l}^{k-1}\po 1 - a_j\pf \\
 & \leq & \exp \po - m\sum_{j=l}^{k-1}  \po \frac{1}{t_j}\pf^{\frac{\theta}{\theta+\eta}}\pf\\
  & \leq & \exp \po - \frac{m}{\sqrt d} \int_{t_l}^{t_k}  \po \frac{1}{s}\pf^{\frac{\theta}{\theta+\eta}} ds\pf.
 \end{eqnarray*}
 Since we choose $t > 2 t_0 + \sqrt d$ there exists $l\in\mathbb N$ such that $t_l \geq \frac t2 \geq t_l - \sqrt d$, hence
 \begin{eqnarray*}
 \|\mu P_{0,t } - \nu_{t }\|_{TV} & = & d_n\\
 &\leq & m' t_{l}^{-\frac{\eta}{\theta+\eta}} + \exp \po - m'' \po t^{\frac{\eta}{\theta+\eta}} - t_l^{\frac{\eta}{\theta+\eta}}\pf \pf\\
 &\leq & M t^{-\frac{\eta}{\theta+\eta}}
 \end{eqnarray*}
 for some $m'',M>0$.
 
Let $(X,Y)$ be a velocity jump process with generator $(L_s)_{s\geq 0}$ and initial law $\mu$ and $(V,W)$ be a r.v. with law $\nu_t$. By Laplace method for any $\alpha>1$ there exists a constant $C$ such that for all $h>0$,
\begin{eqnarray*}
\mathbb P\po U(V) > \min U + h\pf & \leq & Ce^{-\frac{h}{\alpha\varepsilon_t}}.
\end{eqnarray*}
The conclusion follows from
\begin{eqnarray*}
\mathbb P\po U(X_t) > \min U + h\pf & \leq & \mathbb P\po U(V) > \min U + h\pf + \|\mu  P_{0,t } - \nu_{t }\|_{TV}\\
 & \leq & Ce^{-\frac{h}{\alpha\varepsilon_t}} + M t^{-\frac{\eta}{\theta+\eta}}.
\end{eqnarray*}
\end{proof}

\textbf{Remark:} Here is an heuristic to adapt the previous arguments to prove Theorem \ref{ThmDimDRecuit} holds with $\theta = E_*$: given any smooth path $\gamma:[0,1] \rightarrow \mathbb T^d$ with $|\gamma'|=1$ and any $\delta>0$, a process $W$ with generator \eqref{EqGeneDimDU0} (namely for which $U=0$) starting from $\gamma(0)$ has a positive probability (depending on $r,d$ and $\gamma$) to reach $\gamma(1)$ in a given time (say $t=2$) staying in a tube of diameter $\delta$ around the path $\gamma$. Here, a tube is defined by
\[\mathcal T_{\gamma,\delta} =\{(x,y)\in\mathbb T^d\times\Sd\text{ s.t. }\exists s\in[0,1],\ |x-\gamma(s)|+|y-\gamma'(s)|\leq \delta \}.\]
 Then we can define a process $Z$ with generator \eqref{EqGenerateurDimD} to be equal to $W$ up to a time at which a clock depending on the cooling schedule (similar to $S_1$ and $S_2$ in the proof of Theorem \ref{ThmDimdCouple}) rings, at which point $Z$ is reflected. If $\gamma$ is an optimal path to leave a cusp $C$, in the sense a path with $\max \gamma - \min \gamma$ equal to $D$ the depth of $C$, the process is not reflected during this climb with probability of order $\exp(-(D+f(\delta))/\varepsilon)$ where $f(\delta)$ goes to 0 with $\delta$.

This leads to an Eyring–Kramers formula (with a non explicit prefactor since the probability for $W$ to approximatively follow $\gamma$ is not easily tractable) and more generally to an adaptation of all the arguments from Section \ref{SectionInhomogene}. Note however that such a study in the regime $\varepsilon\rightarrow 0$ only quantifies metastability due to energy barriers, so that $E_*$ naturally appears, but not entropic barriers (see \cite{LelievreMetastable} for the distinction). The latter are indeed related to the very probability for $W$ to follow a particular path, and a really relevant question to compare the velocity Gibbs sampler to the Fokker-Planck diffusion is: is it easier (i.e. faster) to find a particular path, in some applied problem, with broken lines rather than with a Brownian motion ? When there are both energy and entropic barriers, is there an optimal amount of inertia to overcome both~?
\bibliographystyle{plain}
\bibliography{biblio}

\begin{thebibliography}{10}

\bibitem{logSob}
C.~An\'e, S.~Blach\`ere, D.~{Chafa\"i}, P~Foug\`eres, I.~Gentil, F.~Malrieu,
  C.~Roberto, and G.~Scheffer.
\newblock {\em Sur les in\'egalit\'es de Sobolev logarithmiques}.
\newblock Panoramas et synth\`eses. Soci\'et\'e math\'ematique de France,
  Paris, 2000.

\bibitem{Krell}
R.~{Aza{\"i}s}, J.-B. {Bardet}, A.~{Genadot}, N.~{Krell}, and P.-A. {Zitt}.
\newblock {Piecewise deterministic Markov process - recent results}.
\newblock {\em ArXiv e-prints}, September 2013.

\bibitem{HitandRun}
C.~J.~P. B{\'e}lisle, H.~E. Romeijn, and R.~L. Smith.
\newblock Hit-and-run algorithms for generating multivariate distributions.
\newblock {\em Math. Oper. Res.}, 18(2):255--266, 1993.

\bibitem{BLBMZ2}
M.~{Bena{\"i}m}, S.~{Le Borgne}, F.~{Malrieu}, and P.-A. {Zitt}.
\newblock Quantitative ergodicity for some switched dynamical systems.
\newblock {\em Electron. Commun. Probab.}, 17:no. 56, 14, 2012.

\bibitem{Bovier1}
A.~Bovier, M.~Eckhoff, V.~Gayrard, and M.~Klein.
\newblock Metastability in reversible diffusion processes. {I}. {S}harp
  asymptotics for capacities and exit times.
\newblock {\em J. Eur. Math. Soc. (JEMS)}, 6(4):399--424, 2004.

\bibitem{Bovier2}
A.~Bovier, V.~Gayrard, and M.~Klein.
\newblock Metastability in reversible diffusion processes. {II}. {P}recise
  asymptotics for small eigenvalues.
\newblock {\em J. Eur. Math. Soc. (JEMS)}, 7(1):69--99, 2005.

\bibitem{Calvez}
V.~Calvez, G.~Raoul, and C.~Schmeiser.
\newblock Confinement by biased velocity jumps: aggregation of {\it
  {e}scherichia coli}.
\newblock {\em Kinet. Relat. Models}, 8(4):651--666, 2015.

\bibitem{Catoni}
O.~Catoni.
\newblock Rough large deviation estimates for simulated annealing: application
  to exponential schedules.
\newblock {\em Ann. Probab.}, 20(3):1109--1146, 1992.

\bibitem{DiaconisMiclo}
P.~Diaconis and L.~Miclo.
\newblock On the spectral analysis of second-order {M}arkov chains.
\newblock {\em Ann. Fac. Sci. Toulouse Math. (6)}, 22(3):573--621, 2013.

\bibitem{Othmer}
R.~Erban and H.G. Othmer.
\newblock From individual to collective behavior in bacterial chemotaxis.
\newblock {\em SIAM J. Appl. Math.}, 65(2):361--391 (electronic), 2004/05.

\bibitem{Fontbona2010}
J.~Fontbona, H.~Gu{\'e}rin, and F.~Malrieu.
\newblock Quantitative estimates for the long-time behavior of an ergodic
  variant of the telegraph process.
\newblock {\em Adv. in Appl. Probab.}, 44(4):977--994, 2012.

\bibitem{FreidlinWentzell}
M.~I. Freidlin and A.~D. Wentzell.
\newblock {\em Random perturbations of dynamical systems}, volume 260 of {\em
  Grundlehren der Mathematischen Wissenschaften [Fundamental Principles of
  Mathematical Sciences]}.
\newblock Springer, Heidelberg, third edition, 2012.
\newblock Translated from the 1979 Russian original by Joseph Sz{\"u}cs.

\bibitem{GadatPanloup2013}
S.~Gadat and F.~Panloup.
\newblock Long time behaviour and stationary regime of memory gradient
  diffusions.
\newblock {\em Annales de l'Institut Henri Poincaré, to appear}, 2013.

\bibitem{Hajek}
B.~Hajek.
\newblock Cooling schedules for optimal annealing.
\newblock {\em Math. Oper. Res.}, 13(2):311--329, 1988.

\bibitem{Holley}
R.~A. Holley, S.~Kusuoka, and D.~W. Stroock.
\newblock Asymptotics of the spectral gap with applications to the theory of
  simulated annealing.
\newblock {\em J. Funct. Anal.}, 83(2):333--347, 1989.

\bibitem{Kalivas}
J.H. {Kalivas}.
\newblock {\em Adaption of Simulated Annealing to Chemical Optimization
  Problems}.
\newblock Data Handling in Science and Technology. Elsevier, 1995.

\bibitem{LelievreMetastable}
T.~Leli{\'e}vre.
\newblock Two mathematical tools to analyze metastable stochastic processes.
\newblock pages 791--810, 2013.

\bibitem{Lelievre2012}
T.~{Leli{\`e}vre}, F.~{Nier}, and G.~A. {Pavliotis}.
\newblock {Optimal non-reversible linear drift for the convergence to
  equilibrium of a diffusion}.
\newblock {\em Journal of Statistical Physics, to appear}, December 2012.

\bibitem{Miclo92}
L.~Miclo.
\newblock Recuit simul\'e sur {$R^n$}. \'{E}tude de l'\'evolution de
  l'\'energie libre.
\newblock {\em Ann. Inst. H. Poincar\'e Probab. Statist.}, 28(2):235--266,
  1992.

\bibitem{Volte-Face}
L.~{Miclo} and P.~{Monmarch{\'e}}.
\newblock {Étude spectrale minutieuse de processus moins indécis que les
  autres}.
\newblock {\em Lecture Notes in Mathematics}, 2078:459.

\bibitem{Monmarche2013}
P.~Monmarch{\'e}.
\newblock Hypocoercive relaxation to equilibrium for some kinetic models.
\newblock {\em Kinet. Relat. Models}, 7(2):341--360, 2014.

\bibitem{MonmarcheRecuitHypo}
P.~{Monmarch{\'e}}.
\newblock {Hypocoercivity in metastable settings and kinetic simulated
  annealing}.
\newblock {\em ArXiv e-prints}, February 2015.

\bibitem{Ratanov}
N.~{Ratanov}.
\newblock Telegraph processes with random jumps and complete market models.
\newblock {\em Methodol. Comput. Appl. Probab.}, 17(3):677--695, 2015.

\bibitem{Lelievre2006}
A.~{Scemama}, T.~{Leli{\`e}vre}, G.~{Stoltz}, and M.~{Caffarel}.
\newblock An efficient sampling algorithm for variational monte carlo.
\newblock {\em Journal of Chemical Physics}, 125, September 2006.

\end{thebibliography}
\end{document}